\documentclass[a4paper,11pt]{article}

\usepackage{amsmath}
\usepackage{amssymb}
\usepackage{amsthm}
\usepackage[mathscr]{eucal}
\usepackage{enumerate}

\usepackage[a4paper]{geometry}

\usepackage{graphicx}
\usepackage{tikz}
\usetikzlibrary{shapes,patterns}

\usepackage{subcaption}
\usepackage{url}

\usepackage{multirow}
\usepackage{pdflscape}
\usepackage{afterpage}

\usepackage{math-standard}
\usepackage{math-sequences}
\usepackage{ut-probability}

\usepackage{ulem}

\newcommand{\esum}{\ensuremath{\sum_{e \in E}}}
\newcommand{\vsum}{\ensuremath{\sum_{v \in V}}}
\newcommand{\gout}{\gamma_{+}}
\newcommand{\gin}{\gamma_{-}}
\newcommand{\din}{D^{-}}
\newcommand{\dout}{D^{+}}
\newcommand{\rdegree}{\mathcal{D}}
\newcommand{\rdin}{\rdegree^{-}}
\newcommand{\rdout}{\rdegree^{+}}
\newcommand{\pearson}{r}
\newcommand{\brho}{\ensuremath{\hat{\pearson}}}
\newcommand{\spearman}{\rho}
\newcommand{\spearmanaverage}{\overline{\spearman}}
\newcommand{\kendall}{\tau}

\newtheorem{defn}{Definition}[section]
\newtheorem{lem}[defn]{Lemma}

\newtheorem{thm}[defn]{Theorem}

\title{Degree-degree dependencies in directed networks with heavy-tailed degrees}
\author{Pim van der Hoorn\footnote{University of Twente, w.l.f.vanderhoorn@utwente.nl}, 
				Nelly Litvak\footnote{University of Twente, n.litvak@utwente.nl}}
\date{\today}

\begin{document}

\normalem

\maketitle
\begin{abstract}
In network theory, Pearson's correlation coefficients are most commonly used to measure the degree 
assortativity of a network. We investigate the behavior of these coefficients in the setting of 
directed networks with heavy-tailed degree sequences. We prove that for graphs where the in- and 
out-degree sequences satisfy a power law with realistic parameters, Pearson's correlation 
coefficients converge to a non-negative number in the infinite network size limit. We propose 
alternative measures for degree-degree dependencies in directed networks based on Spearman's rho 
and Kendall's tau. Using examples and calculations on the Wikipedia graphs for nine different
languages, we show why these rank correlation measures are more suited for measuring degree 
assortativity in directed graphs with heavy-tailed degrees. 
\end{abstract}


\section{Introduction}

In the analysis of the topology of complex networks a feature that is often studied is the degree-degree
dependency, also called degree assortativity of the network. A network is called assortative, when nodes 
with high degree have a preference to be connected to nodes of similar large degree. When nodes with 
large degree have a connection preference for nodes with low degree the network is said to be 
disassortative. A measure for degree assortativity was first given for undirected networks by Newman
\cite{Newman2002}, which corresponds to Pearson's correlation coefficient of the degrees at the ends of 
a random edge in the network. A similar definition for directed networks was introduced in~\cite{Newman2003} 
and later adopted for analysis of directed complex networks in~\cite{Foster2010} and~\cite{Piraveenan2012}. 

Degree assortativity in networks has been analyzed in a variety of scientific fields such as neuroscience, 
molecular biology, information theory and social network sciences and has been found to influence several 
properties of a network. In~\cite{Kaltenbrunner2011} and \cite{Laniado2011} degree-degree correlations are 
used to investigate the structure of collaboration networks of a social news sharing website and Wikipedia 
discussion pages, respectively. Neural networks with high assortativity seem to behave more efficiently 
under the influence of noise~\cite{Franciscis2011} and information content has been shown to depend on the 
absolute value of the degree assortativity~\cite{Piraveenan2009}. The effects of degree-degree dependencies 
on epidemic spreading have been studied in percolation theory \cite{Boguna2002,Vazquez2003} and it has been 
shown, for instance, that the epidemic threshold depends on these correlations. Degree assortativity 
is used in the analysis of networks under attack, e.g. P2P networks \cite{Srivastava2012,srivastava2011}. 
Networks with high degree assortativity seem to be less stable under attack,~\cite{Brede2005}. In the case of 
directed networks, recent research~\cite{liu2014impact} has shown that degree-degree dependencies can 
influence the rate of consensus in directed social networks like Twitter.

Recently it has been shown~\cite{Litvak2012,Litvak2013} that for undirected networks of which the degree 
sequence satisfies a power law distribution with exponent $\gamma \in (1, 3)$, Pearson's correlation 
coefficient scales with the network size, converging to a non-negative number in the infinite network size 
limit. Because most real world networks have been reported to be scale free with exponent in $(1, 3)$, 
c.f.~\cite[Table II]{Albert2002, Newman2003a}, this could then explain why large networks are rarely 
classified as disassortative. In~\cite{Litvak2012,Litvak2013} a new measure, corresponding to Spearman's 
rho~\cite{Spearman1904}, has been proposed as an alternative.

In this paper we will extend the analysis in~\cite{Litvak2012} to the setting of directed networks. Here we
have to consider four types of degree-degree dependencies, depending on the choice for in- or out-degree on 
either side of an edge. Our message is, similar to that of~\cite{Litvak2012}, that Pearson's correlation 
coefficients are size biased and produce undesirable results, hence we should look for other means to measure 
degree-degree dependencies. 

We consider networks where the in- and out-degree sequences have a power law distribution. We will give 
conditions on the exponents of the in- and out-degree sequences for which the assortativity measures defined 
in~\cite{Foster2010} and~\cite{Piraveenan2012} converge to a non-negative number in the infinite network size 
limit. This result is a strong argument against the use of Pearson's correlation coefficients for measuring 
degree-degree dependencies in such directed networks. To strengthen this argument we also give examples which 
clearly show that the values given by Pearson's correlation coefficients do not represent the true dependency 
between the degrees, which it is supposed to measure. As an alternative we propose correlation measures based 
on Spearman's rho~\cite{Spearman1904} and Kendall's tau~\cite{Kendall1938}. These measures are based on the 
ranking of the degrees rather than their value and hence do not exhibit the size bias observed in Pearson's 
correlation coefficients. We will give several examples where the difference between these three measures is 
shown. We also include an example for which one of the four Pearson's correlation coefficients converges to a 
random variable in the infinite network size limit and therefore will obviously produce uninformative results. 
Finally we calculate all four degree-degree correlations on the Wikipedia network for nine different languages 
using all the assortativity measures proposed in this paper.

This paper is structured as follows. In Section~\ref{sec:definitions} we introduce notations. Pearson's 
correlation coefficients are introduced in Section~\ref{sec:pearson} and a convergence theorem is given for 
these measures. We introduce the rank correlations Spearman's rho and Kendall's tau for degree-degree dependencies
in Section~\ref{sec:rankcorrelations}. Example graphs that illustrate the difference between the three 
measures are presented in Section~\ref{sec:examples} and the degree-degree correlations for the Wikipedia 
graphs are presented in Section~\ref{sec:experiments}. Finally, in Section~\ref{sec:discussion} we briefly 
discuss the results and there interpretation.

\section{Definitions and notations}\label{sec:definitions}

We start with the formal definition of the problem and introduce the notations that will be used 
throughout the paper.
 
\subsection{Graphs, vertices and degrees}\label{ssec:graphsvertices}

We will denote by $G = (V, E)$ a directed graph with vertex set $V$ and edge set $E \subseteq V \times 
V$. For an edge $e \in E$, we denote its source by $e_\ast$ and its target by $e^\ast$. With each 
directed graph we associate two functions $D^+, D^- : V \to \N$ where $D^+(v) := |\{e \in E | e_\ast = 
v\}|$ is the out-degree of the vertex $v$ and $D^-(v) := |\{e \in E | e^\ast = v\}|$ the in-degree. 
When considering sequences of graphs, we denote by $G_n = (V_n, E_n)$ an element of the sequence 
$\seq{G}$. We will further use subscripts to distinguish between the different graphs in the sequence. 
For instance, $\dout_n$ and $\din_n$ will denote the out- and in-degree functions of the graph $G_n$, 
respectively. 

\subsection{Four types of degree-degree dependencies}\label{ssec:fourcorrelations}

In this paper we are interested in measuring dependencies between the degrees at both sides of an edge. 
That is, we measure the relation between two vectors $X$ and $Y$ as a function of the edges $e \in E$ 
corresponding to the degrees of $e_\ast$ and $e^\ast$, respectively. In the undirected case this is 
called the degree assortativity. In the directed setting however, we can consider any combination of 
the two degree types resulting in four types of degree-degree dependencies, illustrated in Figure 
\ref{fig:fourcorrelations}.  

From Figure~\ref{fig:fourcorrelations} one can already observe some interesting features of these 
dependencies. For instance, in the Out/In case the edge that we consider contributes to the degrees 
on both sides. We will later see that for this reason the Out/In dependency in fact generalizes the 
undirected case. More precisely, our result for the Out/In dependencies generalizes the result from 
\cite{Litvak2013} when we transform from the undirected to the directed case by making every edge 
bi-directional.
 
For the other three dependency types we observe that there is always at least one side where the 
considered edge does not contribute towards the degree on that side. We will later see that for these 
dependency types the dependency of the in- and out-degree of a vertex will play a role.

\tikzstyle{vertex}=[fill, circle, minimum size=4pt]
\tikzstyle{edge}=[style=thick, color=black]

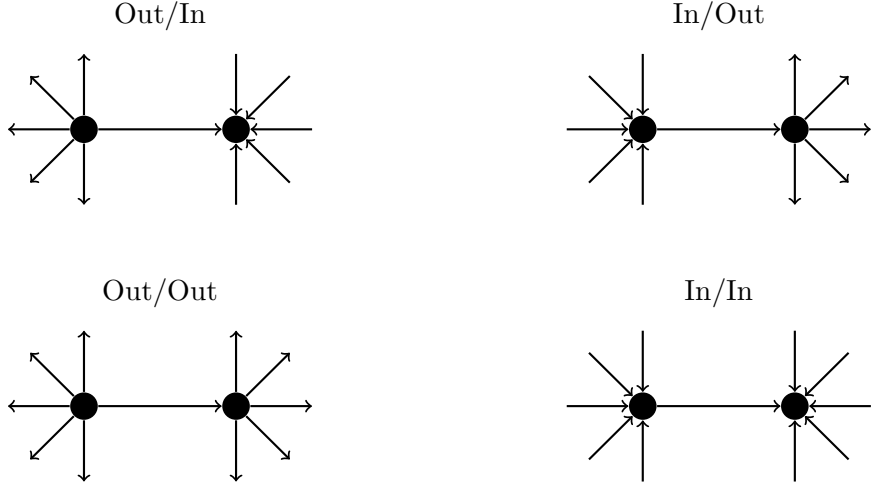
\begin{figure}%
	\begin{subfigure}{0.5\textwidth}
		\centering
		\begin{tikzpicture}
			\draw node[vertex] (left) at (1,0) {};
			\draw node[vertex] (right) at (3,0) {};
			\draw [->][edge] (left) -- (right);
			\draw [->][edge] (left) -- (1,1);
			\draw [->][edge] (left) -- (0.293,0.707);
			\draw [->][edge] (left) -- (0,0);
			\draw [->][edge] (left) -- (0.293,-0.707);
			\draw [->][edge] (left) -- (1,-1);
			\draw [->][edge] (3,1) -- (right);
			\draw [->][edge] ((3.707,0.707) -- (right);
			\draw [->][edge] (4,0) -- (right);
			\draw [->][edge] (3.707,-0.707) -- (right);
			\draw [->][edge] (3,-1) -- (right);
			\draw node at (2,1.5) {Out/In};
		\end{tikzpicture}
	\end{subfigure}%
	\begin{subfigure}{0.5\textwidth}
		\centering
		\begin{tikzpicture}
			\draw node[vertex] (left) at (1,0) {};
			\draw node[vertex] (right) at (3,0) {};
			\draw [->][edge] (left) -- (right);
			\draw [->][edge] (1,1) -- (left);
			\draw [->][edge] (0.293,0.707) -- (left);
			\draw [->][edge] (0,0) -- (left);
			\draw [->][edge] (0.293,-0.707) -- (left);
			\draw [->][edge] (1,-1) -- (left);
			\draw [->][edge] (right) -- (3,1);
			\draw [->][edge] (right) -- (3.707,0.707);
			\draw [->][edge] (right) -- (4,0);
			\draw [->][edge] (right) -- (3.707,-0.707);
			\draw [->][edge] (right) -- (3,-1);
			\draw node at (2,1.5) {In/Out};
		\end{tikzpicture}
	\end{subfigure}

	\vspace{0.8cm}

	\begin{subfigure}{0.5\textwidth}
		\centering
		\begin{tikzpicture}
			\draw node[vertex] (left) at (1,0) {};
			\draw node[vertex] (right) at (3,0) {};
			\draw [->][edge] (left) -- (right);
			\draw [->][edge] (left) -- (1,1);
			\draw [->][edge] (left) -- (0.293,0.707);
			\draw [->][edge] (left) -- (0,0);
			\draw [->][edge] (left) -- (0.293,-0.707);
			\draw [->][edge] (left) -- (1,-1);
			\draw [->][edge] (right) -- (3,1);
			\draw [->][edge] (right) -- (3.707,0.707);
			\draw [->][edge] (right) -- (4,0);
			\draw [->][edge] (right) -- (3.707,-0.707);
			\draw [->][edge] (right) -- (3,-1);
			\draw node at (2,1.5) {Out/Out};
		\end{tikzpicture}
	\end{subfigure}%
	\begin{subfigure}{0.5\textwidth}
		\centering
		\begin{tikzpicture}
			\draw node[vertex] (left) at (1,0) {};
			\draw node[vertex] (right) at (3,0) {};
			\draw [->][edge] (left) -- (right);
			\draw [->][edge] (1,1) -- (left);
			\draw [->][edge] (0.293,0.707) -- (left);
			\draw [->][edge] (0,0) -- (left);
			\draw [->][edge] (0.293,-0.707) -- (left);
			\draw [->][edge] (1,-1) -- (left);
			\draw [->][edge] (3,1) -- (right);
			\draw [->][edge] ((3.707,0.707) -- (right);
			\draw [->][edge] (4,0) -- (right);
			\draw [->][edge] (3.707,-0.707) -- (right);
			\draw [->][edge] (3,-1) -- (right);
			\draw node at (2,1.5) {In/In};
		\end{tikzpicture}
	\end{subfigure}

\caption{Four degree-degree dependency types}%
\label{fig:fourcorrelations}%
\end{figure}

\section{Pearson's correlation coefficient}\label{sec:pearson}

Among degree-degree dependency measures, the measure proposed by Newman~\cite{Newman2002,Newman2003} has 
been widely used. This measure is the statistical estimator for the Pearson correlation coefficient of 
the degrees on both sides of a random edge. However, for undirected networks with heavy tailed degrees 
with exponent $\gamma \in (1, 3)$ it was proved \cite{Litvak2013} that this measure converges, in the 
infinite size network limit, to a non-negative number. Therefore, in these cases, Pearson's correlation 
coefficient is not able to correctly measure negative degree-degree dependencies. In this section we will 
extend this result to directed networks proving that also here Pearson's correlation coefficients are not 
the right tool to measure degree-degree dependencies. 

Let us consider Pearson's correlation coefficients as in \cite{Newman2002,Newman2003}, adjusted to the 
setting of directed graphs as in~\cite{Foster2010,Piraveenan2012}. This will constitute four formulas 
which we combine into one. Take $\alpha, \beta \in \{+, -\}$, that is, we let $\alpha$ and $\beta$ index 
the type of degree (out- or in-degree). Then we get the following expression for the four Pearson's 
correlation coefficients:

\begin{equation}
		\pearson_\alpha^\beta(G) = \frac{1}{\sigma_\alpha(G) \sigma^\beta(G)}\left( \frac{1}{|E|} 
		\esum D^\alpha(e_\ast)D^\beta(e^\ast) - \frac{1}{|E|^2}\esum D^\alpha(e_\ast) \esum D^\beta(e^\ast)\right)
		\label{eq:negpearsonedge},
\end{equation}
where
\begin{align}
		&\sigma_\alpha(G) = \sqrt{\frac{1}{|E|} \esum D^\alpha(e_\ast)^2 
		- \frac{1}{|E|^2}\left( \esum D^\alpha(e_\ast) \right)^2} \text{ and} \label{eq:sigalphaedge} \\
		&\sigma^\beta(G) = \sqrt{\frac{1}{|E|} \esum D^\beta(e^\ast)^2 
		- \frac{1}{|E|^2}\left( \esum D^\beta(e^\ast) \right)^2}. \label{eq:sigbetaedge}
\end{align}
Here we utilize the notations for the source and target of an edge by letting the superscript index 
denote the specific degree type of the target $e^\ast$ and the subscript index the degree type of the 
source $e_\ast$. For instance $\pearson^-_+$ denotes the Pearson correlation coefficient for the Out/In 
relation. 

It is convenient to rewrite the summations over edges to summations over vertices by observing that 
\[
	\esum D^\alpha(e_\ast)^k = \vsum \dout D^\alpha(v)^k
\] 
and similarly 
\[
	\esum D^\alpha(e^\ast)^k = \vsum \din D^\alpha(v)^k
\] for all $k > 0$. Plugging this into~\eqref{eq:negpearsonedge}-\eqref{eq:sigbetaedge} we arrive at 
the following definition.

\begin{defn}\label{def:pearson}
Let $G = (V, E)$ be a directed graph and let $\alpha, \beta \in \{+, -\}$. Then the Pearson's $\alpha$
-$\beta$ correlation coefficient is defined by
\begin{equation}
	\pearson_\alpha^\beta(G) = \frac{1}{\sigma_\alpha(G) \sigma^\beta(G)} \frac{1}{|E|} 
		\esum D^\alpha(e_\ast)D^\beta(e^\ast) - \brho_\alpha^\beta(G) \label{eq:pearsonedges},
\end{equation}
where
\begin{align}
		&\brho_\alpha^\beta(G) = \frac{1}{\sigma_\alpha(G) \sigma^\beta(G)} \frac{1}{|E|^2} 
		\vsum \dout(v) D^\alpha(v) \vsum \din(v) D^\beta(v), \label{eq:negpearsonvertex} \\
		&\sigma_\alpha(G) = \sqrt{\frac{1}{|E|} \vsum \dout(v) D^\alpha(v)^2 
		- \frac{1}{|E|^2}\left( \vsum \dout(v) D^\alpha(v) \right)^2}, \label{eq:alphavariance} \\
		&\sigma^\beta(G) = \sqrt{\frac{1}{|E|} \vsum \din(v) D^\beta(v)^2 
		- \frac{1}{|E|^2}\left( \vsum \din(v) D^\beta(v) \right)^2}. \label{eq:betavariance}
\end{align}
\end{defn}

Just as in the undirected case, c.f. \cite{Litvak2012, Litvak2013}, the wiring of the network only 
contributes to the positive part of~\eqref{eq:pearsonedges}. All other terms are completely determined 
by the in- and out-degree sequences. This fact enables us to analyze the behavior of $\pearson_\alpha
^\beta(G)$, see Section~\ref{ssec:convergence}. Observe also that in contrast to undirected graphs, in 
the directed case the correlation between the in- and out-degrees of a vertex can play a role, take 
for instance $\alpha = -$ and $\beta = +$.

Note that in general $\pearson_\alpha^\beta(G)$ might not be well defined, for either $\sigma_\alpha(G)$ 
or $\sigma^\beta(G)$ might be zero, for example, when $G$ is a directed cyclic graph of arbitrary size. 
From equations~\eqref{eq:sigalphaedge} and~\eqref{eq:sigbetaedge} it follows that $\sigma_\alpha(G)$ and 
$\sigma^\beta(G)$ are the variances of $X$ and $Y$, where $X = D^\alpha(e_\ast)$ and $Y = D^\beta(e^\ast)$, 
$e\in E$, with probability $1/|E|$. Thus, $\sigma_\alpha(G) \ne 0$ is only possible if $D^\alpha(v) \ne 
D^\alpha(w)$ for some $v, w \in V$. Moreover, $v$ and $w$ must have non-zero out-degree for at least one 
such pair $v,w$, so that $D^\alpha(v)$ and $D^\alpha(w)$ are counted when we traverse over edges. This 
argument is formalized in the next lemma, which provides necessary and sufficient conditions so that 
$\sigma_\alpha(G)$, $\sigma^\beta(G) \ne 0$.

\begin{lem}\label{lem:degreesumbounds}
Let $G = (V, E)$ be a graph and take $\alpha, \beta \in \{+, -\}$. Then the following holds:
\begin{equation}
		\frac{1}{|E|}\left(\sum_{v \in V} D^\alpha(v)D^\beta(v)\right)^2 
		\le \sum_{v \in V} D^\alpha(v)D^\beta(v)^2
	\label{eq:degreesumbound}
\end{equation}
and strict inequality holds if and only if there exits distinct $v, w \in V$ such that $D^\alpha(v)$, $D^\alpha(w) > 0$ 
and $D^\beta(v) \ne D^\beta(w)$.
\end{lem}

\begin{proof}
Recall that $|E| = \sum_{v \in V} D^\alpha(v)$ for any $\alpha \in \{+, -\}$. Then we have:
\begin{align*}
	&|E|\sum_{v \in V} D^\alpha(v)D^\beta(v)^2 - \left(\sum_{v \in V} D^\alpha(v)D^\beta(v)\right)^2 \\
	&= \sum_{w \in V} \sum_{v \in V\setminus w} D^\alpha(w)D^\alpha(v)D^\beta(v)^2 
	- D^\alpha(w)D^\beta(w)D^\alpha(v)D^\beta(v) \\
	&= \frac{1}{2}\sum_{w \in V} \sum_{v \in V\setminus w} D^\alpha(w)D^\alpha(v)\left(
	D^\beta(w)^2 - 2D^\beta(w)D^\beta(v) + D^\beta(v)^2 \right) \\
	&= \frac{1}{2}\sum_{w \in V} \sum_{v \in V\setminus w} D^\alpha(w)D^\alpha(v)
	\left(D^\beta(w) - D^\beta(v)\right)^2 \ge 0,
\end{align*}
which proves~\eqref{eq:degreesumbound}. From the last line one easily sees that strict inequality holds
if and only if there exits distinct $v, w \in V$ such that $D^\alpha(v)$, $D^\alpha(w) > 0$ and 
$D^\beta(v) \ne D^\beta(w)$.
\end{proof}   

\subsection{Convergence of Pearson's correlation coefficients}\label{ssec:convergence}

In this section we will prove that Pearson's correlation coefficients~\eqref{eq:pearsonedges}, 
calculated on sequences of growing graphs satisfying rather general conditions, converge to a 
non-negative value. We start by recalling the definition of big theta.

\begin{defn}\label{def:bigtheta}
Let $f, g : \N \to \R_{> 0}$ be positive functions. Then $f = \Theta(g)$ if there exist $k_1, k_2 \in 
\R_{> 0}$ and an $N \in \N$ such that for all $n \ge N$
\[
	k_1 g(n) \le f(n) \le k_2 g(n).
\]
When we have two sequences $\seq{a}$ and $\seq{b}$ we write $a_n = \Theta(b_n)$ for $\seq{a} = \Theta(\seq{b})$.
\end{defn}

Next, we will provide the conditions that our sequence of graphs needs to satisfy and prove the result. 
These conditions are based on properties of i.i.d. sequences of regularly varying random variables, 
which are often used to model scale-free distributions. We will provide a more thorough motivation of 
the chosen conditions in Section~\ref{ssec:motivation}. From here on we denote by $x \vee y$ and $x 
\wedge y$ the maximum and minimum of $x$ and $y$, respectively. 

\begin{defn}\label{def:graphsequencespace}
	For $\gin, \gout \in \R_{> 0}$ we denote by $\mathscr{G}_{\gin \gout}$ the space of all sequences of graphs 
	$\seq{G}$ with the following properties:
	\begin{enumerate}[\upshape G1]
		\item $|V_n| = n$.
		\item There exists a $N \in \N$ such that for all $n \ge N$ there exist $v, w \in V_n$ with 
		$D_n^\alpha(v)$, $D_n^\alpha(w) > 0$ and $D_n^\alpha(v) \ne D_n^\alpha(w)$, for all 
		$\alpha \in \{+, -\}$.
		\item
		For all $p, q \in \R_{> 0}$,
		\begin{equation*}
			\sum_{v \in V_n} D_n^+(v)^p D^-_n(v)^q = \Theta(n^{p/\gout \vee q/\gin \vee 1}).
		\end{equation*}
		\item
		For all $p, q \in \R_{> 0}$, if $p < \gout$ and $q < \gin$ then 
		\begin{equation*}
			\nlim \frac{1}{n}\sum_{v \in V_n} D_n^+(v)^p D^-_n(v)^q := d(p, q) \in (0, \infty).
		\end{equation*}
		Where the limits are such that for all $a, b \in \N$, $k, m > 1$ with $1/k + 1/m = 1$, 
		$a + p < \gout$ and $b + q < \gin$ we have,
		\[
			d(a, b)^{\frac{1}{m}}d(p, q)^{\frac{1}{k}} > d(\frac{a}{m} + \frac{p}{k}, \frac{b}{m} + \frac{q}{k}).
		\]
	\end{enumerate}
\end{defn}

Now we are ready to give the convergence theorem for Pearson's correlation coefficients, Definition~\ref{def:pearson}.

\begin{thm}\label{thm:convergencepearson}
	Let $\alpha, \beta \in \{+, -\}$. Then there exists an area $A_\alpha^\beta \subseteq \R^2$ such that 
	for $(\gout, \gin) \in A_\alpha^\beta$ and $\seq{G} \in \mathscr{G}_{\gin \gout}$,
	\[
		\nlim \brho^\beta_\alpha(G_n) = 0
	\]
	and hence any limit point of $\pearson^\beta_\alpha(G_n)$ is non-negative.
\end{thm}

\begin{proof}
Let $\seq{G}$ be an arbitrary sequence of graphs. It is clear that if $\brho^\beta_\alpha(G_n) \to 0$ 
then any limit point of $\pearson^\beta_\alpha(G_n)$ is non-negative. Therefore we only need to prove 
the first statement. To this end we define the following sequences,
\begin{align*}
a_n &= \frac{1}{|E_n|} \left(\sum_{v \in V_n} \dout_n(v)D_n^\alpha(v)\right)^2,
&&b_n = \frac{1}{|E_n|} \left(\sum_{v \in V_n} \din_n(v)D_n^\beta(v)\right)^2, \\
c_n &= \sum_{v \in V_n} \dout_n(v)D_n^\alpha(v)^2,
&&d_n = \sum_{v \in V_n} \din_n(v)D_n^\beta(v)^2, \\
\end{align*}
and observe that $\brho_\alpha^\beta(G_n)^2 = a_n b_n/(c_n - a_n)(d_n - b_n)$. Now if $\seq{G} \in 
\mathscr{G}_{\gin \gout}$ then because of G2 and Lemma~\ref{lem:degreesumbounds} there exists an $N \in 
\N$ such that for all $n \ge N$ we have $c_n > a_n$ and $d_n > b_n$, so $\brho^\beta_\alpha(G_n)$ is well-defined
for all $n \ge N$. Next, using G3, we get that	$a_n = \Theta(n^a)$, $b_n = \Theta(n^b)$, $c_n = \Theta(n^c)$ 
and $d_n = \Theta(n^d)$ for certain constants $a$, $b$, $c$ and $d$, which depend on $\gin, \gout$ and the
degree-degree correlation type chosen. Because 
$\brho_\alpha^\beta(G_n) \to 0$ if and only if $\brho_\alpha^\beta(G_n)^2 \to 0$, we need to find sufficient 
conditions for which $a_n b_n/(c_n - a_n)(d_n - b_n) \to 0$. It is clear that either $a < c$ and 
$b_n/(d_n - b_n)$ is bounded or $b < d$ and $a_n/(c_n - a_n)$ is bounded are sufficient.
It turns out that this is exactly the case when either $a < c$ and $b \le d$ or $a \le c$ and $b < d$. We 
will do the analysis for the In/Out degree-degree correlation. The analysis for the other three 
correlation types is similar. Figure \ref{fig:convergenceareas} shows all four areas $A^\beta_\alpha$. 

When $\alpha = -$ and $\beta = +$ we get the following constants
\begin{align*}
 a, b &= 2 \left(\frac{1}{\gout} \vee \frac{1}{\gin} \vee 1\right) - 1 \\
 c &= \left(\frac{1}{\gout} \vee \frac{2}{\gin} \vee 1\right) \\
 d &= \left(\frac{2}{\gout} \vee \frac{1}{\gin} \vee 1\right)
\end{align*}
It is clear that when $1 < \gin, \gout < 2$ then $a < c$ and $b < d$ and hence $\brho^\beta_\alpha \to 0$.
Now if $1 < \gin < 2$ and $\gout \ge 2$ then $a = b = d = 1 < c$. Using G4 we get that
$\nlim d_n/n = d(2, 1)$ and
\begin{align*}
	\nlim \frac{b_n}{n} &= \nlim \frac{\left(\sum_{v \in V_n} \din_n(v)\dout_n(v)\right)^2}{n^2} \frac{n}{|E_n|}\\
	&= \nlim \left(\frac{\sum_{v \in V_n} \din_n(v)\dout_n(v)}{n}\right)^2 
	\left(\frac{\sum_{v \in V_n} \din_n(v)}{n}\right)^{-1}\\
	&= \frac{d(1, 1)^2}{d(0, 1)} < d(2, 1) = \nlim \frac{d_n}{n},
\end{align*}
where, for the last part, we again used G4. From this it follows that $b_n/(d_n - b_n)$ is bounded and so 
$\brho^\beta_\alpha \to 0$. A similar argument applies to the case $\gin \ge 2$ and $1 < \gout < 2$, where the only 
difference is that $a = b = c = 1 < d$, hence 
\[
	A^+_- = \{(x, y) \in \R | 1 < x < 2, \quad y > 1 \} \cup \{(x, y) \in \R | 1 < y < 2, \quad x > 1 \}.
\]
Using similar arguments, we obtain:
\begin{align*}
	A_+^- &= \{(x, y) \in \R^2 | 1 < x < 3, \quad y > 1\} \cup \{(x, y) \in \R^2 | 1 < y < 3, \quad x > 1\}, \\
	A_+^+ &= \{(x, y) \in \R^2 | 1 < x < 3, \quad y > 1\} \text{ and}\\
	A_-^- &= \{(x, y) \in \R^2 | 1 < y < 3, \quad x > 1\}.
\end{align*}
\end{proof}

\begin{figure}[ht]
\centering
\begin{tikzpicture}[scale=0.8]
			\draw (0,0) -- (0,6) node[above] {$\gin$};
			\draw (0,0) -- (6,0) node[right] {$\gout$};
			\draw[fill=black,opacity=0.6] ((1,1) -- (6,1) -- (6,3) -- (3,3) -- (3,6) -- (1,6) -- (1,1);
			\draw (1,0.1) -- (1,-0.1) node[below](1,-0.1) {1};
			\draw (0.1,1) -- (-0.1,1) node[left](-0.1,1) {1};
			\draw (3,0.1) -- (3,-0.1) node[below](3,-0.1){3};
			\draw (0.1,3) -- (-0.1,3) node[left](-0.1,3){3};
			\draw node at (6,6) {\Large $A^-_+$};
\end{tikzpicture}\quad\begin{tikzpicture}[scale=0.8]
		\draw (0,0) -- (0,6) node[above] {$\gamma_{-}$};
		\draw (0,0) -- (6,0) node[right] {$\gamma_{+}$};
		\draw[fill=black,opacity=0.6] (1,1) -- (1,6) -- (2,6) -- (2,2) -- (6,2) 
		-- (6,1) -- (1,1);
		\draw (2,0.1) -- (2,-0.1) node[below](2,-0.1){2};
		\draw (0.1,2) -- (-0.1,2) node[left](-0.1,2){2};
		\draw (1,0.1) -- (1,-0.1) node[below](1,-0.1) {1};
		\draw (0.1,1) -- (-0.1,1) node[left](-0.1,1) {1};
		\draw node at (6,6) {\Large $A^+_-$};
\end{tikzpicture}
\begin{tikzpicture}[scale=0.8]
		\draw (0,0) -- (0,6) node[above] {$\gamma_{-}$};
		\draw (0,0) -- (6,0) node[right] {$\gamma_{+}$};
		\draw[fill=black,opacity=0.6] (1,1) -- (1,6) -- (3,6) -- (3,1) -- (1,1); 
		\draw (1,0.1) -- (1,-0.1) node[below]{1};
		\draw (3,0.1) -- (3,-0.1) node[below]{3};
		\draw (0.1,1) -- (-0.1,1) node[left]{1};
		\draw node at (6,6) {\Large $A^+_+$};
\end{tikzpicture}\quad\begin{tikzpicture}[scale=0.8]
			\draw (0,0) -- (0,6) node[above] {$\gamma_{-}$};
			\draw (0,0) -- (6,0) node[right] {$\gamma_{+}$};
			\draw[fill=black,opacity=0.6] (1,1) -- (6,1) -- (6,3) -- (1,3) -- (1,1);
			\draw (0.1,1) -- (-0.1,1) node[left]{1};
			\draw (0.1,3) -- (-0.1,3) node[left]{3};
			\draw (1,0.1) -- (1,-0.1) node[below]{1};
			\draw node at (6,6) {\Large $A_-^-$};
\end{tikzpicture}
\caption{Four areas $A^\beta_\alpha$, where $r^\beta_\alpha$ converges to a non-negative number.}%
\label{fig:convergenceareas}%
\end{figure}
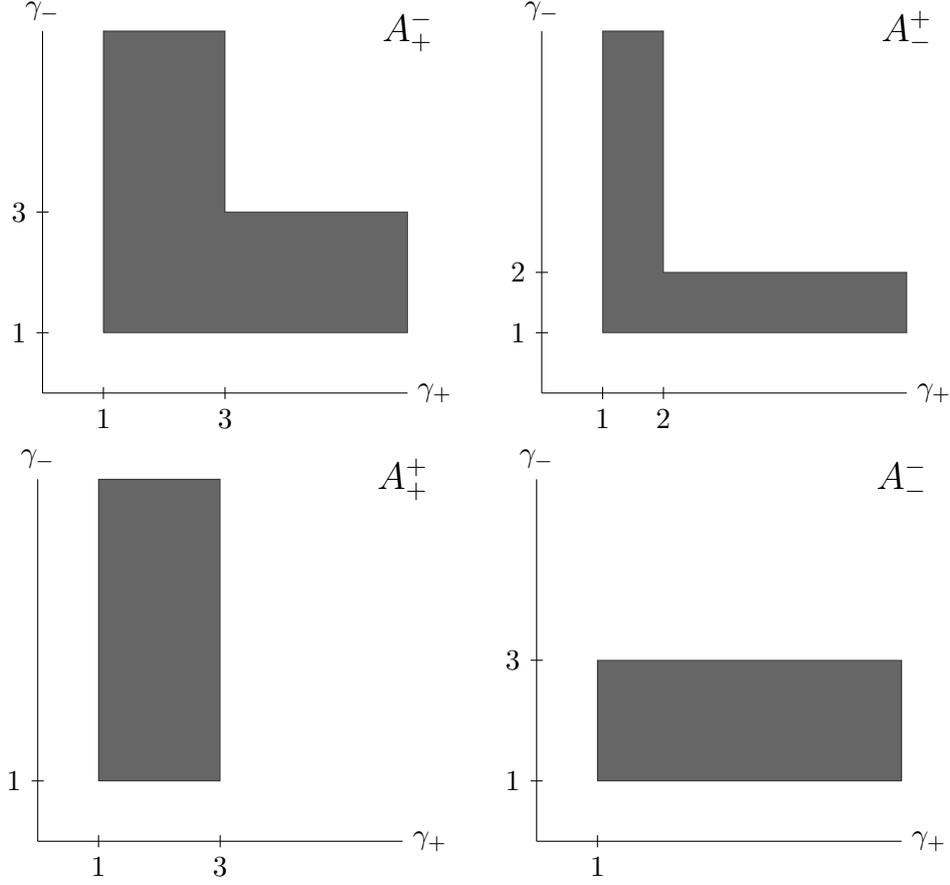

Let us now provide an intuitive explanation for the areas $A_\alpha^\beta$, as depicted in Figure
\ref{fig:convergenceareas}. The key observation is that due to G3 the terms with the highest power 
of either $\dout_n$ or $\din_n$ will dominate in $\brho^\beta_\alpha(G_n)$. Therefore, if these 
moments do not exist, then the denominator will grow at a larger rate then the numerator, hence 
$\brho^\beta_\alpha \to 0$.

Taking $\alpha = + = \beta$, we see that $\din$ only has terms of order one while $\dout$ has terms
up to order three. This explains why $A^+_+ = \{(x, y) \in \R | 1 < x \le 3, y > 1\}$. Area $A_-^-$ is then
easily explained by observing that the expression for $\pearson^-_-(G)$ is obtained from $\pearson^+_+(G)$ 
by interchanging $\dout$ and $\din$.

For the Out/In correlation, i.e. $\alpha = +$ and $\beta = -$, we see from 
equations~\eqref{eq:negpearsonvertex}-\eqref{eq:betavariance} that $\brho^-_+(G)$ splits into a product of two 
terms, each completely determined by either in- or out-degrees,
\[
	\frac{\frac{1}{|E|}\sum_{v \in V} D^\alpha(v)^2}
	{\sqrt{\frac{1}{|E|}\sum_{v \in V} D^\alpha(v)^3 - \frac{1}{|E|^2}\left(\sum_{v \in V} D^\alpha(v)^2\right)^2}},
\]
with $\alpha \in \{+, -\}$.
These terms are of the exact same form as the expression in~\cite{Litvak2012} for the undirected degree-degree
correlation. Because both $\dout$ and $\din$ have terms of order three, one sees that 
\[
	A_+^- = \{(x, y) \in \R^2 | 1 < x < 3, \quad y > 1\} \cup \{(x, y) \in \R^2 | 1 < y < 3, \quad x > 1\}.
\] 
Now take a undirected network and make it directed by replacing each undirected edge with a 
bi-directional edge. Then $\dout(v) = \din(v)$ for all $v \in V$ and hence $\pearson^-_+(G)$ equals the 
expression of equation (3.4) in~\cite{Litvak2012} when we replace $D$ by either $\dout$ or $\din$. 

Theorem~\ref{thm:convergencepearson} has several consequences. First of all, no matter what mechanism 
is used for generating networks, if the conditions of the theorem are satisfied then for large enough 
networks the degree-degree correlations will always be non-negative. This could explain why in most 
large networks strong disassortativity has not been registered. We will present such examples in Section
\ref{sec:examples}. Second, if the underlying model that governs the topology of the network is in line 
with the conditions of the theorem, then one cannot compare networks of different sizes that arise from 
this model. For in this case, the degree-degree correlation coefficients $\pearson^\beta_\alpha$ will 
decrease with the network size.

\subsection{Motivation for $\mathscr{G}_{\gin \gout}$}\label{ssec:motivation}

In this section we will motivate Definition~\ref{def:graphsequencespace}. G1 is easily motivated, 
for we want to consider infinite network size limits. G2 combined with Lemma~\ref{lem:degreesumbounds} 
ensures that from a certain grah size $N$, $\pearson^\beta_\alpha(G_n)$ is always well-defined. 
Conditions G3 and G4 are related to heavy-tailed degree sequences that are modeled using regularly 
varying random variables.

A random variable $X$ is called regularly varying with exponent $\gamma$ if for all $t>0$, $\Prob(X > t) 
= L(t)t^{-\gamma}$ for some slowly varying function $L$, that is $\lim_{t \to \infty} L(tx)/L(t) = 1$ 
for all $x>0$. We write $\mathcal{R}_{-\gamma}$ for the class of all such distribution functions and
write $X \in \mathcal{R}_{-\gamma}$ to denote a regularly varying random variable with exponent $\gamma$. 
For such a random variable $X$ we have that $\Exp{X^p} < \infty$ for all $0 < 
p < \gamma$. 

Through experiments it has been shown that many real world networks, both directed and undirected, 
have degree sequences whose distribution closely resembles a power law distribution, c.f. Table II 
of~\cite{Albert2002} and~\cite{Newman2003a}. Suppose we take two random variables $\rdout \in 
\mathcal{R}_{\gout}$, $\rdin \in \mathcal{R}_{\gin}$ and consider, for each $n$, the degree sequences 
$(D^\pm_n(v))_{v \in V_n}$ as i.i.d. copies of these random variables. Then for all $0 < p <\gout$ and
$0 < q < \gin$
\[
	\nlim \frac{1}{n}\sum_{v \in V_n} D_n^+(v)^p D^-_n(v)^q = \Exp{(\rdout)^p(\rdin)^q}.
\]
Moreover, since $\rdegree^\pm$ is non-degenerate, we have $\Exp{\left(\rdegree^\pm\right)^k} > 
\Exp{\rdegree^\pm}^k$, and thus  by taking $d(p, q) = \Exp{(\rdout)^p(\rdin)^q}$, we get G4 where the
second part follows from H\"older's inequality. Although i.i.d. sequences generated by sampling from 
in- and out-degree distributions do not in general constitute a graphical sequence, it is often the 
case that one can modify this sequence into a graphical sequence preserving i.i.d. properties 
asymptotically. Consider for example \cite{chen2013}, where a directed version of the configuration
model is introduced and it is proven (Theorem 2.4) that the degree sequences are asymptotically 
independent.

The property G3 is associated with the scaling of the sums
$\sum_{v \in V_n} D_n^+(v)^p D^-_n(v)^q$ and is related to the central limit theorem for regularly 
varying random variables. When we model the degrees as i.i.d. copies of independent regularly varying random variables 
$\rdegree^+ \in \mathcal{R}_{-\gout}$, $\rdegree^- \in \mathcal{R}_{-\gin}$ and take $p \ge \gout$ or $q \ge \gin$ 
then  $\sum_{v \in V_n} D_n^+(v)^p D^-_n(v)^q$ is in the domain of attraction of a $\gamma$-stable random 
variable $S(\gamma)$, where $\gamma = (\gout/p \wedge \gin/q)$, c.f. \cite{Cline1986}. This means that 
\begin{equation}
	\frac{1}{a_n} \sum_{v \in V_n} D_n^+(v)^p D^-_n(v)^q \dlim S(\gout/p \wedge \gin/q),\quad \mbox{as $n\to\infty$}
	\label{eq:centrallimit}
\end{equation}
for some sequence $a_n = \Theta(n^{q/\gin \vee p/\gout})$, where $\dlim$ denotes convergence in distribution. 
Informally, one could say that $\sum_{v \in V_n} D_n^+(v)^p D^-_n(v)^q$ scales as $n^{q/\gin \vee p/\gout}$ when 
either the $p$ or $q$ moment does not exist and as $n$ when both moments exist, hence, $\sum_{v \in V_n} D_n^+(v)^p 
D^-_n(v)^q$ scales as $n^{q/\gin \vee p/\gout \vee 1}$, which is what G3 states. 
For completeness we include the next lemma, which shows that \eqref{eq:centrallimit} implies that G3 holds with 
high probability. 

We remark that although the motivation for G3 is based on results where the regularly varying random 
variables are assumed to be independent the dependent case can be included. For this, one needs 
to adjust the scaling parameters in G3 for the specified dependence. In our numerical experiments 
the in- and out- degrees in Wikipedia graphs show strong independence, hence G3  holds for networks 
such as Wikipedia.
{\sloppy

}
\begin{lem}\label{lem:probconvergencetheta}
Let $\seq{X}$ be a sequence of positive random variables such that
\[
	\frac{X_n}{a_n} \dlim X, \quad \mbox{as $n\to\infty$},
\]
for some sequence $\seq{a}$ and positive random variable $X$. Then for each $0 < \varepsilon < 1$, there exists an 
$N_\varepsilon \in \N$ and $\kappa_\varepsilon \ge \ell_\varepsilon > 0$ such that for all 
$n \ge N_\varepsilon$
\[
	\Prob(\ell_\varepsilon a_n \le X_n \le \kappa_\varepsilon a_n) \ge 1 - \varepsilon.
\]
\end{lem}

\begin{proof}
Let $0 < \varepsilon < 1$ and take $\delta > 0$, $0 < \ell \le \kappa$ such that
\[
	\Prob(\ell \le X \le \kappa) \ge 1 - \varepsilon + \delta.
\]
Then, because  $X_n/a_n \dlim X$ as $n \to \infty$, there exists an $N \in \N$ such that for all $n \ge N$,
\[
	|\Prob(\ell \le X \le \kappa) - \Prob(\ell a_n \le X_n \le \kappa a_n)| < \delta.
\]
Now we get for all $n \ge N$,
\[
	1 - \varepsilon + \delta - \Prob(\ell a_n \le X_n \le \kappa a_n) 
	\le \Prob(\ell \le X \le \kappa) - \Prob(\ell a_n \le X_n \le \kappa a_n)
	\le \delta,
\]
hence $\Prob(\ell a_n \le X_n \le \kappa a_n) \ge 1 - \varepsilon$.
\end{proof}

\section{Rank correlations}\label{sec:rankcorrelations}

In this section we consider two other measures for degree-degree dependencies, Spearman's rho
\cite{Spearman1904} and Kendall's tau~\cite{Kendall1938}, which are based on the rankings of the degrees 
rather than their actual value. We will define these dependency measures and argue that they do not have 
unwanted behavior as we observed for Pearson's correlation coefficients. We will later use examples to 
enforce this argument and show that Spearman's rho and Kendall's tau are better candidates for measuring 
degree-degree dependencies.

\subsection{Spearman's rho}

Spearman's rho~\cite{Spearman1904} is defined as the Pearson correlation coefficient of the vector of ranks.  
Let $G = (V, E)$ be a directed graph and $\alpha, \beta \in \{+, -\}$. In order to adjust the definition of 
Spearman's rho to the setting of directed graphs we need to rank the vectors $(D^\alpha(e_\ast))_{e \in E}$
and $(D^\beta(e^\ast))_{e \in E}$. These will, however, in general have many tied values. For instance, suppose 
that $D^\alpha(v) = m$ for some $v \in V$, then edges $e \in E$ with $e_\ast = v$ satisfy $D^\alpha(e_\ast) = 
D^\alpha(v)$. Therefore, we will encounter the value $D^\alpha(v)$ at least $m$ times in the
vector $(D^\alpha(e_\ast))_{e \in E}$. We will consider two strategies for resolving ties: uniformly at random 
(Section~\ref{ssec:spearmanuniform}), and using an average ranking scheme (Section~\ref{ssec:spearmanaverage}).

\subsubsection{Resolving ties uniformly at random}\label{ssec:spearmanuniform}

Given a sequence $\{x_i\}_{1 \le i \le n}$ of distinct elements in $\R$ we denote by $R(x_j)$ the rank of $x_j$, 
i.e. $R(x_j) = |\{i | x_i \ge x_j\}|$, $1\le j\le n$. The definition of Spearman's rho in the setting of directed 
graphs is then as follows.

\begin{defn}\label{def:spearmanuniform}
Let $G = (V, E)$ be a directed graph, $\alpha, \beta \in \{+, -\}$ and let $(U_e)_{e \in E}$, $(W_e)_{e \in E}$ 
be i.i.d. copies of independent uniform random variables $U$ and $W$ on $(0, 1)$, respectively. Then
we define the $\alpha$-$\beta$ Spearman's rho of the graph $G$ as
\begin{equation}
	\spearman^\beta_\alpha(G) = \frac{12 \esum R^\alpha(e_\ast)R^\beta(e^\ast) - 3|E|(|E| + 1)^2}
	{|E|^3 - |E|},
\label{eq:spearmanuniform}
\end{equation}
where $R^\alpha(e_\ast) = R(D^\alpha(e_\ast) + U_e)$ and $R^\beta(e^\ast) = R(D^\beta(e^\ast) + W_e)$.
\end{defn}

From~\eqref{eq:spearmanuniform} we see that the negative part of $\spearman^\beta_\alpha(G)$ depends only on the 
number of edges
\[
	\frac{3(|E| + 1)^2}{(|E|^2 - 1)} = 3 + \frac{6|E| + 4}{|E|^2 - 1},
\] 
while for $\pearson^\beta_\alpha(G)$ it depended on the values of the degrees, see Definition~\ref{def:pearson}. 
When $\seq{G} \in \mathscr{G}_{\gout, \gin}$, with $\gout, \gin > 1$ then it follows that $|E_n| = \theta(n)$ hence
$3 + (6|E| + 4)/(|E|^2 - 1) \to 3$, as $n \to \infty$. Therefore we see that the negative contribution will always 
be at least $3$ and so $\spearman^\beta_\alpha(G_n)$ does not in general converge to a non-negative number while 
$\pearson^\beta_\alpha(G_n)$ does.

When calculating $\spearman^\beta_\alpha(G)$ on a graph $G$ one has to be careful, for each instance will give 
different ranks of the tied values. This could potentially give rise to very different results among several instances, 
see Section~\ref{sssec:spearman} for an example. Therefore, in experiments, we will take an average of 
$\spearman^\beta_\alpha(G)$ over several instances of the uniform ranking.

\subsubsection{Resolving ties with average ranking}\label{ssec:spearmanaverage}

A different approach for resolving ties is to assign the same average rank to all tied values. Consider, 
for example, the sequence $(1, 2, 1, 3, 3)$. Here the two values  of $3$ have ranks $1$ and $2$, but instead we 
assign the rank $3/2$ to both of them. With this scheme the sequence of ranks becomes $(9/2, 3, 9/2, 3/2, 3/2)$. 
This procedure can be formalized as follows.  
\begin{defn}\label{def:averagerank}
Let $(x_i)_{1 \le i \le n}$ be a sequence in $\R$ then we define the average rank of an element $x_i$ as
\[
	\overline{R}(x_i) = |\{j | x_j > x_i\}| + \frac{|\{j | x_j = x_i\}| + 1}{2}.
\]
\end{defn}
 
Observe that in the above definition the total average rank is preserved: $\sum_{i = 1}^n \overline{R}
(x_i) = n(n + 1)/2$. The difference with resolving ties uniformly at random is that we in general do 
not know $\sum_{i = 1}^n \overline{R}(x_i)^2$, for this depends on how many ties we have for each 
value. We now define the corresponding version of Spearman's rho of graphs as follows.

\begin{defn}\label{def:spearmanaverage}
let $G = (V, E)$ be a directed graph, $\alpha, \beta \in \{+, -\}$ and denote by $\overline{R}^\alpha
(e_\ast)$ and $\overline{R}^\beta(e^\ast)$ the average ranks of $D^\alpha(e_\ast)$ among $(D^\alpha
(e_\ast))_{e \in E}$ and $D^\beta(e^\ast)$ among $(D^\beta(e^\ast))_{e \in E}$, respectively. Then we 
define the $\alpha$-$\beta$ Spearman's rho with average resolution of ties by
\begin{equation}
	\spearmanaverage^\beta_\alpha(G) = 
	\frac{4 \esum \overline{R}^\alpha(e_\ast) \overline{R}^\beta(e^\ast) - |E|(|E| + 1)^2}
	{\overline{\sigma}_\alpha(G)\overline{\sigma}^\beta(G)},
\label{eq:spearmanaverage}
\end{equation}
where
\begin{align*}
	\overline{\sigma}_\alpha(G) &= \sqrt{4\vphantom{^\beta}\esum \overline{R}^\alpha(e_\ast)^2 
		- |E|(|E| + 1)^2}\\
	\intertext{and}
	\overline{\sigma}^\beta(G) &= \sqrt{4\esum \overline{R}^\beta(e^\ast)^2 - |E|(|E| + 1)^2}.
\end{align*}
\end{defn}

Note that $\spearmanaverage^\beta_\alpha(G)$ does not suffer from any randomness in the ranking of the 
degrees. Hence, in contrast to \eqref{eq:spearmanuniform}, here we do not need to take an average over 
multiple instances. The next lemma shows that taking the expectation over the uniform ranking is 
actually equal to applying the average ranking scheme.

\begin{lem}\label{lem:uniformaverage}
	Let $G = (V, E)$ be a graph, $e \in E$ and $\alpha, \beta \in \{+, -\}$. Then
	\begin{enumerate}[\upshape i)] 
		\item 
		$\Exp{R^\alpha(e_\ast)} = \overline{R}^\alpha(e_\ast), \quad 
		\Exp{R^\beta(e^\ast)} = \overline{R}^\beta(e^\ast), \quad \text{and}$
		\item 
		$\Exp{R^\alpha(e_\ast)R^\beta(e^\ast)} = \overline{R}^\alpha(e_\ast) \overline{R}^\beta(e^\ast)$
	\end{enumerate}
\end{lem}
	
\begin{proof}
	\hfill
	\begin{enumerate}[\upshape i)]
		\item 
		We will only prove the first statement. The proof for the second one is similar. Since $R^\alpha
		(e_\ast) = R(D^\alpha (e_\ast)) + U_e$ and $(U_e)_{e \in E}$ are i.i.d. copies of 
		an uniform random	variable $U$ on $(0, 1)$ we have that
		\begin{align*}
			&\sum_{f \in E} I\left\{D^\alpha(f_\ast) = D^\alpha(e_\ast)\right\} \Exp{I\left\{U_f \ge U_e\right\}} \\
			&= \sum_{f \in E} I\left\{D^\alpha(f_\ast) = D^\alpha(e_\ast)\right\} \left(I\left\{f = e\right\}
				+ \frac{1}{2}I\left\{f \ne e\right\}\right) \\
			&= \frac{1}{2}\sum_{f \in E} I\left\{D^\alpha(f_\ast) = D^\alpha(e_\ast)\right\} + \frac{1}{2}.
		\end{align*}
		It follows that
		\begin{align*}
			\Exp{R^\alpha(e_\ast)} &= \Exp{\sum_{f \in E} I\left\{D^\alpha(f_\ast) + U_f \ge D^\alpha(e_\ast) + U_e\right\}} \\
			&= \sum_{f \in E} I\left\{D^\alpha(f_\ast) > D^\alpha(e_\ast)\right\} 
				+ \sum_{f \in E} I\left\{D^\alpha(f_\ast) = D^\alpha(e_\ast)\right\} \Exp{I\left\{U_f \ge U_e\right\}} \\
			&= \sum_{f \in E} I\left\{D^\alpha(f_\ast) > D^\alpha(e_\ast)\right\} 
				+ \frac{1}{2}\sum_{f \in E} I\left\{D^\alpha(f_\ast) = D^\alpha(e_\ast)\right\} + \frac{1}{2} \\
			&= \overline{R}^\alpha(e_\ast).
		\end{align*}
		\item
		By definition we have that
		\begin{align*}
			R^\alpha&(e_\ast)R^\beta(e^\ast)= \sum_{f, g \in E}I\left\{\vphantom{D^\beta}D^\alpha(f_\ast) 
				> D^\alpha(e_\ast)\right\}I\left\{D^\beta(g^\ast) 
				> D^\beta(e^\ast)\right\} \\
			&\hspace{10pt}+ \sum_{f, g \in E}I\left\{\vphantom{D^\beta} D^\alpha(f_\ast) > D^\alpha(e_\ast)\right\}
				I\left\{D^\beta(g^\ast) = D^\beta(e^\ast)\right\}I\left\{\vphantom{D^\beta} W_g \ge W_e\right\} \\
			&\hspace{10pt}+ \sum_{f, g \in E}I\left\{\vphantom{D^\beta} D^\alpha(f_\ast) = D^\alpha(e_\ast)\right\}
				I\left\{\vphantom{D^\beta} U_f \ge U_e\right\}I\left\{D^\beta(g^\ast) > D^\beta(e^\ast)\right\} \\
			&\hspace{10pt}+ \sum_{f, g \in E} I\left\{\vphantom{D^\beta} D^\alpha(f_\ast) = D^\alpha(e_\ast)\right\}
				I\left\{D^\beta(g^\ast) = D^\beta(e^\ast)\right\}I\left\{U_f \ge U_e\right\}
				I\left\{\vphantom{D^\beta} W_g \ge W_e\right\}.
		\end{align*}
		Therefore, since $(U_f)_{f \in E}$ and $(W_g)_{g \in E}$ are i.i.d. copies of independent uniform 
		random variables $U$ and $W$ on $(0, 1)$, respectively, the result follows by applying i). 
	\end{enumerate}
\end{proof}
	
From Lemma~\ref{lem:uniformaverage} we conclude that instead of calculating $\spearman_\alpha^\beta$ several times
and then taking the average we can immediately apply the average ranking which limits the total calculations to 
just one. Moreover, we have that
\begin{equation} \label{eq:spearmanuniformaverage}
	\Exp{\spearman_\alpha^\beta(G)} = \frac{3\overline{\sigma}_\alpha\overline{\sigma}^\beta}{|E|^3 - |E|}
	\spearmanaverage_\alpha^\beta(G),
\end{equation}
which emphasizes that the difference between the uniform at random and average ranking scheme is determined by the
number of ties in the degrees.

\subsection{Kendall's Tau}
Another common rank correlation is Kendall's tau~\cite{Kendall1938}, which measures the weighted 
difference between the number of concordant and discordant pairs of the joint observations 
$(x_i, y_i)_{1 \le i \le n}$. More precisely, a pair $(x_i, y_i)$ and $(x_j, y_j)$ of joint observations
is concordant if $x_i < x_j$ and $y_i < y_j$ or if $x_i > x_j$ and $y_i > y_j$. They are called discordant
if $x_i < x_j$ and $y_i > y_j$ or if $x_i > x_j$ and $y_i < y_j$. 
\begin{defn}\label{def:kendalltau}
Let $G = (V, E)$ be a directed graph, $\alpha, \beta \in \{-, +\}$ and denote by $\mathscr{N}_c$ and 
$\mathscr{N}_d$, respectively, the number of concordant and discordant pairs among 
$\left(D^\alpha(e_\ast), D^\beta(e^\ast)\right)_{e \in E}$. Then we define the $\alpha$-$\beta$ 
Kendall's tau of $G$ by\sloppy{

}
\[
	\kendall^\beta_\alpha(G) = \frac{2(\mathscr{N}_c - \mathscr{N}_d)}{|E|(|E| - 1)}.
\]
\end{defn}

It might seem at first that $\kendall$ does not suffer from ties. However, note that the numerator of 
$\tau$ includes only strictly concordant and discordant pairs, while the denominator is equal to 
the number of all possible pairs, irregardless of the presence of ties. Hence, when the number of ties is 
large, the denominator may become much larger than the numerator resulting in small, even vanishing in the 
graph size limit, values of $\kendall^\beta_\alpha$. We will provide such example in 
Section~\ref{sec:examples}. Since, as discussed above, the sequences $\left(D^\alpha(e_\ast)\right)_{e \in E}$ 
and $\left(D^\beta(e^\ast)\right)_{e \in E}$ naturally have a large number of ties, we cannot expect 
$\kendall^\beta_\alpha(G)$ to take very large (positive or negative) values. To address this issue an weighted 
extension of Kendall's tau was very recently introduced~\cite{Vigna2014}. This new measure also puts more emphasis
on nodes with large in- or out-degrees.

\section{Bridge graph example}\label{sec:examples}

In this section we will provide a sequence of graphs to illustrate the difference between the four
correlation measures in directed networks. We start with a deterministic sequence and will later 
adapt this to a randomized sequence using regularly varying random variables.

\subsection{A deterministic in-out bridge graph}\label{ssec:bridgegraph}

Let $k, m \in \N_{> 0}$, then we define the bridge graph $G(k, m) = (V(k, m), E(k, m))$, displayed in 
Figure~\ref{fig:bridgegraph1}, as follows:  
\[
	V(k, m) = v \cup w \cup \bigcup_{i = 1}^k v_i \cup \bigcup_{j = 1}^{m} w_j, 
	\quad E(k, m) = g \cup \bigcup_{i = 1}^k e_i \cup \bigcup_{j = 1}^{m} f_j, \text{ where}
\]
\[
	e_i = (v_i, v), \, f_j = (w, w_j) \text{ and } g = (v, w).
\]
It follows that $|E(k, m)| = m + k + 1$. For the degrees of $G(k, m)$ we have:
\begin{align*}
	&D^+(v_i) = 1, &&D^-(v_i) = 0, &&&\text{for all } 1 \le i \le k;\\
	&D^+(w_j) = 0, &&D^-(w_j) = 1, &&&\text{for all } 1 \le j \le m;\\
	&D^+(v) = 1, &&D^-(v) = k, \\
	&D^+(w) = m, &&D^-(w) = 1.
\end{align*}

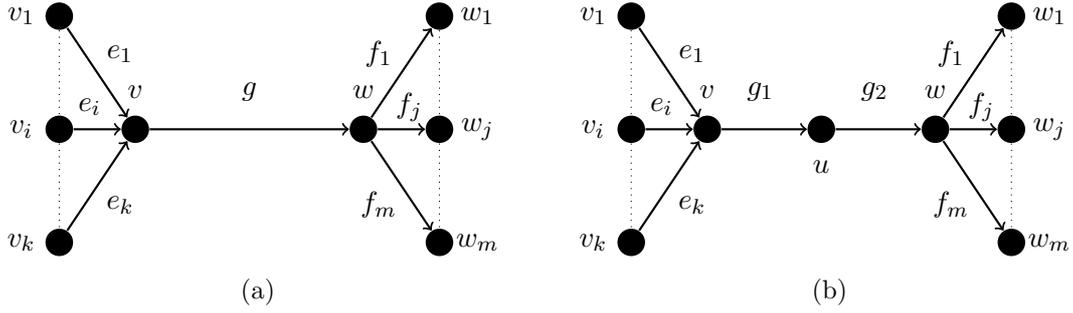
\begin{figure}[t]%
\centering
\begin{subfigure}{0.5\textwidth}
	\centering
	\tikzstyle{vertex}=[fill, circle, minimum size=4pt]
\tikzstyle{edge}=[style=thick, color=black]

\begin{tikzpicture}
	\draw node[vertex] (left) at (1,0) {};
	\draw node[vertex] (right) at (4,0) {};
	\draw node[vertex] (upleft) at (0,1.5) {};
	\draw node[vertex] (ileft) at (0,0) {};
	\draw node[vertex] (downleft) at (0,-1.5) {};
	\draw node[vertex] (iright) at (5,0) {};
	\draw node[vertex] (upright) at (5,1.5) {};
	\draw node[vertex] (downright) at (5,-1.5) {};
	\draw node at (-0.5,1.5) {$v_1$};
	\draw node at (-0.5,0) {$v_i$};
	\draw node at (-0.5,-1.5) {$v_k$};
	\draw node at (0.8,1) {$e_1$};
	\draw node at (0.4,0.3) {$e_i$};
	\draw node at (0.8,-1) {$e_k$};
	\draw node at (1,0.5) {$v$};
	\draw node at (2.5,0.5) {$g$};
	\draw node at (4,0.5) {$w$};
	\draw node at (5.5,1.5) {$w_1$};
	\draw node at (5.5,0) {$w_j$};
	\draw node at (5.5,-1.5) {$w_m$};
	\draw node at (4.2,1) {$f_1$};
	\draw node at (4.6,0.3) {$f_j$};
	\draw node at (4.2,-1) {$f_m$};
	\draw [->][edge] (left) -- (right);
	\draw [->][edge] (upleft) -- (left);
	\draw [->][edge] (ileft) -- (left);
	\draw [->][edge] (downleft) -- (left);
	\draw [->][edge] (right) -- (upright);
	\draw [->][edge] (right) -- (iright);
	\draw [->][edge] (right) -- (downright);
	\draw [dotted] (upleft) -- (ileft);
	\draw [dotted] (ileft) -- (downleft);
	\draw [dotted] (upright) -- (iright);
	\draw [dotted] (iright) -- (downright);
\end{tikzpicture}
	\subcaption{}
	\label{fig:bridgegraph1}
\end{subfigure}~
\begin{subfigure}{0.5\textwidth}
	\centering
	\tikzstyle{vertex}=[fill, circle, minimum size=2pt]
\tikzstyle{edge}=[style=thick, color=black]

\begin{tikzpicture}
	\draw node[vertex] (left) at (1,0) {};
	\draw node[vertex] (mid) at (2.5,0) {};
	\draw node[vertex] (right) at (4,0) {};
	\draw node[vertex] (upleft) at (0,1.5) {};
	\draw node[vertex] (ileft) at (0,0) {};
	\draw node[vertex] (downleft) at (0,-1.5) {};
	\draw node[vertex] (iright) at (5,0) {};
	\draw node[vertex] (upright) at (5,1.5) {};
	\draw node[vertex] (downright) at (5,-1.5) {};
	\draw node at (-0.5,1.5) {$v_1$};
	\draw node at (-0.5,0) {$v_i$};
	\draw node at (-0.5,-1.5) {$v_k$};
	\draw node at (0.8,1) {$e_1$};
	\draw node at (0.4,0.3) {$e_i$};
	\draw node at (0.8,-1) {$e_k$};
	\draw node at (1,0.5) {$v$};
	\draw node at (1.7,0.5) {$g_1$};
	\draw node at (2.5,-0.5) {$u$};
	\draw node at (3.2,0.5) {$g_2$};
	\draw node at (4,0.5) {$w$};
	\draw node at (5.5,1.5) {$w_1$};
	\draw node at (5.5,0) {$w_j$};
	\draw node at (5.5,-1.5) {$w_m$};
	\draw node at (4.2,1) {$f_1$};
	\draw node at (4.6,0.3) {$f_j$};
	\draw node at (4.2,-1) {$f_m$};
	\draw [->][edge] (left) -- (mid);
	\draw [->][edge] (mid) -- (right);
	\draw [->][edge] (upleft) -- (left);
	\draw [->][edge] (ileft) -- (left);
	\draw [->][edge] (downleft) -- (left);
	\draw [->][edge] (right) -- (upright);
	\draw [->][edge] (right) -- (iright);
	\draw [->][edge] (right) -- (downright);
	\draw [dotted] (upleft) -- (ileft);
	\draw [dotted] (ileft) -- (downleft);
	\draw [dotted] (upright) -- (iright);
	\draw [dotted] (iright) -- (downright);
\end{tikzpicture}
	\subcaption{}
	\label{fig:bridgegraph2}
\end{subfigure}
\caption{A graphical representation of the graphs $G(k, m)$ (a) and $\hat{G}(k, m)$ (b).}%
\label{fig:bridgegraphs}%
\end{figure}

Looking at the scatter plot of $(D^-(e_\ast), D^+(e^\ast))_{e \in E(k, m)}$, Figure~\ref{fig:scatterplot1}, 
we see that the point $(k, m)$ contributes towards a positive dependency while the points $(0, 1)$ and 
$(1, 0)$ contribute towards a negative dependency. Hence, depending on how much weight we put on each of 
these points we could argue equally well that this graph could have a positive or negative value for the 
In/Out dependency. We can however extend the in-out bridge graph to a graph for which we do have a clearly 
negative In/Out dependency.

We define the disconnected in-out bridge graph $\hat{G}(k,m) = (\hat{V}(k, m), \hat{E}(k, m))$ from 
$G(k, m)$ by adding a vertex $u$ and replacing the edge $g = (v, w)$ by the edges $g_1 = (v, u)$ and 
$g_2 = (u, w)$, see Figure~\ref{fig:bridgegraph2}. In this graph the node with the largest in-degree, $v$, 
is connected to node $u$, of out-degree 1. Similarly $u$, which has in-degree 1, is connected to the node 
with the highest out-degree, $w$. Therefore we would expect a negative value of In/Out dependency measures. 
This intuition is supported by the scatter plot of $(D^+(e^\ast), D^-(e_\ast))_{e \in \hat{E}(k, m)}$, 
Figure~\ref{fig:scatterplot2}.
{\sloppy

}
Now consider for a fixed $a \in \N$ the sequence of graphs $G_n^a := G(n, an)$ and $\hat{G}_n^a := \hat{G}
(n, an)$. Then, following the above reasoning we would expect any In/Out dependency measure of 
$\hat{G}_n^a$ to converge to -1. 

In Sections~\ref{sssec:pearson} -- \ref{sssc:kendall} we will show that $\nlim \pearson^+_-(\hat{G}_n^a) = 0$ 
while the other three measures indeed yield negative values. Furthermore, we show that $\nlim \pearson^+_-
(G_n^a) = 1$ while $\nlim \spearmanaverage^+_-(G^a_n) = -1$ reflecting the two possibilities for the In/Out 
correlation represented in the scatter plot, Figure~\ref{fig:scatterplot1}. 

\begin{figure}[t]%
\centering
\begin{subfigure}{0.5\textwidth}
	\begin{tikzpicture}[scale=0.7]
	\draw (0,0) -- (0,6); 
	\draw node at (-1,6) {$D^+(e^\ast)$};
	\draw (0,0) -- (6,0);
	\draw node at (6,-1) {$D^-(e_\ast)$};
	\draw (5,0.1) -- (5,-0.1) node[below]{$k$};
	\draw (0.1,5) -- (-0.1,5) node[left]{$m$};
	\draw (1,0) node[below]{1};
	\draw (0,1) node[left]{1};
	\draw node at (1,0) {$\bullet$};
	\draw (1,0) node[above]{$f_j$};
	\draw node at (0,1) {$\bullet$};
	\draw (0,1)node[right]{$e_i$};
	\draw node at (5,5) {$\bullet$};
	\draw (5,5)node[right]{$g$};
\end{tikzpicture}
	\caption{}
	\label{fig:scatterplot1}
\end{subfigure}~
\begin{subfigure}{0.5\textwidth}
	\begin{tikzpicture}[scale=0.7]
	\draw (0,0) -- (0,6); 
	\draw node at (-1,6) {$D^+(e^\ast)$};
	\draw (0,0) -- (6,0);
	\draw node at (6,-1) {$D^-(e_\ast)$};
	\draw (5,0.1) -- (5,-0.1) node[below]{$k$};
	\draw (0.1,5) -- (-0.1,5) node[left]{$m$};
	\draw (1,0) node[below]{1};
	\draw (0,1) node[left]{1};
	\draw node at (1,0) {$\bullet$};
	\draw (1,0) node[above]{$f_j$};
	\draw node at (0,1) {$\bullet$};
	\draw (0,1)node[right]{$e_i$};
	\draw node at (5,1) {$\bullet$};
	\draw (5,1)node[right]{$g_2$};
	\draw node at (1,5) {$\bullet$};
	\draw (1,5)node[right]{$g_1$};
\end{tikzpicture}
	\caption{}
	\label{fig:scatterplot2}
\end{subfigure}
\caption{The scatter plots for the degrees of (a) $G(k, m)$ and (b) $\hat{G}(k, m)$.}%
\label{fig:correlationplot}%
\end{figure}
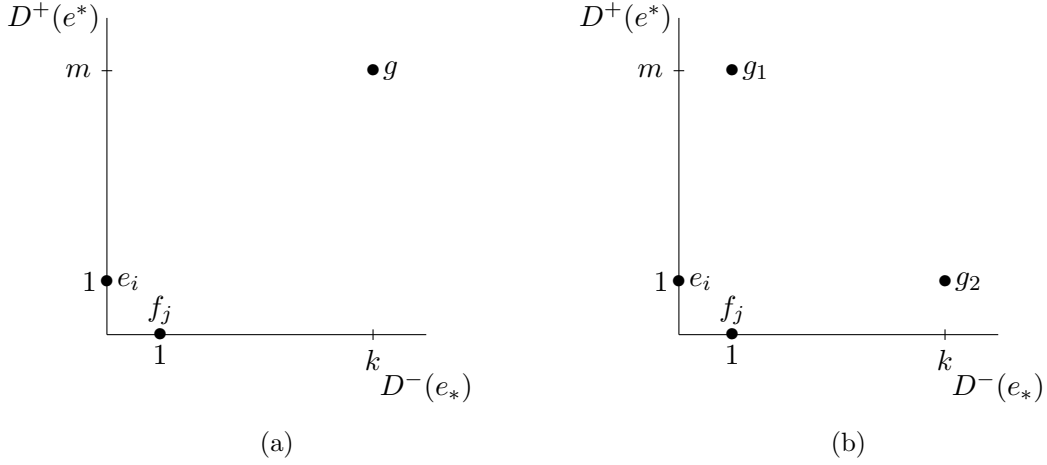

\subsubsection{Pearson In/Out correlation}\label{sssec:pearson}

We start with the graph $G_n^a$. Basic calculations yield that
\begin{align}
	&\sum_{e \in E^a_n} D^-(e_\ast)D^+(e^\ast) = an^2, \label{eq:pearsonbridgeedges} \\
	&\sum_{v \in V^a_n} D^-(v)D^+(v) = (1 + a)n, \label{eq:pearsonbridgevertices} \\
	&\sum_{v \in V^a_n} D^-(v)^2D^+(v) = n^2 + an, \\
	&\sum_{v \in V^a_n} D^-(v)D^+(v)^2 = n + a^2n^2 \label{eq:pearsonbridgelast},
\end{align}
hence, using~\eqref{eq:alphavariance} and~\eqref{eq:betavariance}, we obtain:
\begin{align*}
	|E^a_n|\sigma_-(G^a_n) &= \sqrt{((1 + a)n + 1)(n^2 + an) - (1 + a)^2n^2} \notag \\
	&= \sqrt{(1 + a)n^3 - (n - 1)an}
\end{align*}
and
\begin{align*}
	|E^a_n|\sigma^+(G^a_n) &= \sqrt{((1 + a)n + 1)(n + a^2n^2) - (1 + a)^2n^2} \notag \\
	&= \sqrt{(1 + a)n^3 - (an - 1)n}.
\end{align*}
When we plug this into~\eqref{eq:pearsonedges} with $\alpha = -$ and $\beta = +$ we get
\begin{align}
	\pearson^+_-(G^a_n) &= \frac{|E^a_n|an^2 - (1 + a)^2n^2}
	{|E^a_n|\sigma_\alpha(G_n^a)|E^a_n|\sigma^\beta(G^a_n)} \notag\\ 
	&=\frac{a(1 + a)n^3 - (a^2 + a + 1)n^2}
	{a\sqrt{(1 + a)n^3 - (n - 1)an}\sqrt{(1 + a)n^3 -(an - 1)n}}. \label{eq:pearsonbridgegraph}
\end{align}
From \eqref{eq:pearsonbridgegraph} it follows that if $a \in \N$ is fixed, then $\nlim \pearson^+_-(G^a_n) = 1$, thus 
$\pearson^+_-(G^a_n)$ in fact reflects the connection between $v$ and $w$ where the point $(n, an)$ 
in the scatter plot received the most mass. However, when we turn to $\hat{G}_n^a$ we get a less 
expected result. Splitting the edge $g$ in two adds one to equations
\eqref{eq:pearsonbridgevertices}-\eqref{eq:pearsonbridgelast}, while equation~\eqref{eq:pearsonbridgeedges} 
becomes $(a + 1)n$ which is linear in $n$ instead of quadratic. Because all other terms keep their scale 
with respect to $n$ we easily deduce that for a fixed $a \in \N$, $\nlim \pearson_-^+(\hat{G}_n^a) = 0$. 
This is undesirable for we would expect any In/Out correlation on $\hat{G}^a_n$ to converge to $-1$.

\subsubsection{Spearman In/Out correlation}\label{sssec:spearman}

We start by calculation $\spearmanaverage^+_-(G^a_n)$. For this observe that by~\eqref{eq:spearmanaverage}
and the definition of $G^a_n$ we have that,
\begin{align*}
	&\overline{R}^+((e_{i})^\ast) = 1 + \frac{n + 1}{2}, 
	&&\overline{R}^-((e_i)_\ast) = an + 1 + \frac{n + 1}{2}; \\
	&\overline{R}^+((f_{j})^\ast) = n + 1 + \frac{an + 1}{2}, 
	&&\overline{R}^-((f_j)_\ast) = 1 + \frac{an + 1}{2}; \\
	&\overline{R}^+(g^\ast) = 1,	&&\overline{R}^-(g_\ast) = 1.
\end{align*} 
After some basic calculations we get
\[
	\spearmanaverage^+_-(G^a_n) = \frac{-(a^2 + a)n^3 + (a + 1)^2n^2 + (a + 1)n}
	{(a^2 + a)n^3 + (a + 1)^2n^2 + (a + 1)n} \to -1 \quad \text{as } n \to \infty.
\]
This result is in striking contrast with $\pearson_-^+(G_n^a)$. Indeed, $\spearmanaverage^+_-$
places all the weight on the points $(0, 1)$ and $(1, 0)$. However, based on the scatter plot, see 
Figure~\ref{fig:scatterplot1}, both results could be plausible. 

Let us now compute 
$\spearmanaverage_-^+(\hat{G_n^a})$. For the rankings we have
\begin{align*}
	&\overline{R}^+((e_i)^\ast) = 2 + \frac{n}{2},
	&&\overline{R}^-((e_i)_\ast) = an + 2 + \frac{n + 1}{2}; \\
	&\overline{R}^+((f_j)^\ast) = n + 2 + \frac{an + 1}{2},
	&&\overline{R}^-((f_j)_\ast) = 2 + \frac{an}{2}; \\
	&\overline{R}^+((g_1)^\ast) = 2 + \frac{n}{2}, &&\overline{R}^-((g_1)_\ast) = 1; \\
	&\overline{R}^+((g_2)^\ast) = 1, &&\overline{R}^-((g_2)_\ast) = 2 + \frac{an}{2}.
\end{align*}
Filling this into equation~\eqref{eq:spearmanaverage} we get
\begin{align*}
	\spearmanaverage^+_-(\hat{G}^a_n) &= \frac{-(a^2 + a)n^3 - (a^2 + a)n^2 + (a + 1)n - 2}
	{\bar{\sigma}_-(\hat{G}^a_n)\bar{\sigma}^+(\hat{G}^a_n)},
\end{align*}
where
\begin{align*}
	\bar{\sigma}_-(\hat{G}^a_n) &= \sqrt{(a^2 + a)n^3 + (a^2 + 4a + 2)n^2 + (3a + 4)n - 2} \text{ and} \\
	\bar{\sigma}^+(\hat{G}^a_n) &= \sqrt{(a^2 + a)n^3 + (2a^2 + 4a + 1)n^2 + (4a + 3)n + 2}.
\end{align*}
Because
\[
	\nlim \frac{1}{n^3} \bar{\sigma}_-(\hat{G}^a_n)\bar{\sigma}^+(\hat{G}^a_n) = (a^2 + a)
\]
it follows that 
\[
	\nlim \spearmanaverage^+_-(\hat{G}^a_n) = \nlim \frac{1/n^3}{1/n^3} \left(\frac{-(a^2 + a)n^3 - (a^2 + a)n^2 
	+ (a + 1)n - 2}	{\bar{\sigma}_-(\hat{G}^a_n)\bar{\sigma}^+(\hat{G}^a_n)}\right) = -1,
\]
which equals $\nlim \spearmanaverage_-^+(G_n^a)$. We have already argued that based on the graph and the scatter 
plot we would expect negative In/Out correlation for the sequence $\seq{\hat{G}^a}$. This result is 
in agreement with what we would expect, while $\pearson_-^+(\hat{G}_n^a)$ converges to $0$ as $n \to 
\infty$.

Now we turn to $\spearman^+_-(G^a_n)$. We will show that the choice of ranking of the tied values can 
have a great effect on the outcome of the Spearman's In/Out correlation. In this example we will pick 
two rankings, one will yield $\spearman_-^+(G_n^a)>0$ while the other will give $\spearman_-^+(G_n^a)
< 0$. 

It is clear from the definition of $G^a_n$ that the in- and out-degrees of all $e_i$ are the same, and
this is also true for $f_j$. Let us now impose the following ranking of the vectors $(D^+(e^\ast))_{e 
\in E^a_n}$ and $(D^-(e_\ast))_{e \in E^a_n}$:
\begin{align*}
	&R^+((e_i)^\ast) = an + i, &&R^-((e_i)_\ast) = i, &&&\text{for all } 1 \le i \le n; \\
	&R^+((f_j)^\ast) = j, &&R^-((f_j)_\ast) = n + j, &&&\text{for all } 1 \le j \le an; \\
	&R^+(g^\ast) = 1 + (a + 1)n, &&R^-(g_\ast) = 1 + (a + 1)n.
\end{align*}
Here we ordered the ties by the order of their indices. We calculate that 
\begin{equation}
	\spearman^+_-(G^a_n) = \frac{(a^3 - 3a^2 - 3a + 1)n^3 + 3(a + 1)^2n^2 + 2(a + 1)n}
	{(a^3 + 3a^2 + 3a + 1)n^3 + 3(a + 1)^2n^2 + 2(a + 1)n}.
\label{eq:spearmanbridge1}
\end{equation}
Now let us now order $(D^+(e^\ast))_{e \in E^a_n}$ and $(D^-(e_\ast))_{e \in E^a_n}$ as follows:
\begin{align*}
	&R^+((e_i)^\ast) = (a + 1)n + 1 - i, &&R^-((e_i)_\ast) = i, &&&\text{for all } 1 \le i \le n; \\
	&R^+((f_j)^\ast) = an + 1 - j, &&R^-((f_j)_\ast) = n + j, &&&\text{for all } 1 \le j \le an; \\
	&R^+(g^\ast) = 1 + (a + 1)n, &&R^-(g_\ast) = 1 + (a + 1)n.
\end{align*} 
This order differs from the first one only on the vector $(D^+(e^\ast))_{e \in E^a_n}$, where we now 
ordered the ties based on the reversed order of their indices. Here we get, after some calculations,
\begin{equation}
	\spearman^+_-(G^a_n) = \frac{-(a + 1)^3n^3 + 3(a + 1)^2n^2 + 2(a + 1)n}
	{(a + 1)^3n^3 + 3(a + 1)^2n^2 + 2(a + 1)n}
\label{eq:spearmanbridge2}
\end{equation}
When we compare~\eqref{eq:spearmanbridge2} with~\eqref{eq:spearmanbridge1} we see that for the former 
$\nlim \spearman^+_-(G^a_n) = -1$ for all $a \in \N$ while for the latter we have $\nlim \spearman^+_-
(G^a_n) = (a^3 - 3a^2 - 3a + 1)/(a + 1)^3$. This means that increasing $a$ will actually increase the 
limit of~\eqref{eq:spearmanbridge1}, which becomes positive when $a \ge 4$. If we denote by $d_n^a$ 
the absolute value of the difference between~\eqref{eq:spearmanbridge1} and~\eqref{eq:spearmanbridge2}, 
we get that $\lim_{n \to \infty} d_n^a = 2(a^3 + 1)/(a + 1)^3$ which converges to 2 as $a \to \infty$. 
This agrees with the fact that for~\eqref{eq:spearmanbridge1} it holds that $\lim_{a \to \infty} 
\lim_{n \to \infty} \spearman_-^+(G^a_n) = 1$ while $\lim_{a \to \infty} \lim_{n \to \infty} 
\spearman_-^+(G^a_n) = -1$ for~\eqref{eq:spearmanbridge2}. We see that changing the order of the ties 
can have a large impact on the value of $\spearman^\beta_\alpha(G)$, as was already mentioned in 
Section~\ref{ssec:spearmanuniform}. 
Now, using equation~\eqref{eq:spearmanuniformaverage}, $\lim_{n \to \infty} \spearmanaverage_-^+(G^a_n) 
= -1$ and the fact that
\[
	\overline{\sigma}_\alpha(G^a_n)\overline{\sigma}^\beta(G^a_n) = (a^2 + a)n^3 + (a + 1)^2n^2 + (a + 1)n,
\]
we get that $\lim_{n \to \infty} \Exp{\spearman_-^+(G^a_n)} = -2a/(a + 1)^2$. Notice that, unlike 
$\spearmanaverage_-^+(G^a_n)$, this result still depends on $a$ and converges to 0 as $a \to \infty$. 
This is not unexpected because the majority of edges produce ties, hence, most of the ranks are defined by 
independent realizations of $U$ and $W$. These results indicate that Spearman's rho with average 
resolution of ties is the most informative correlation for this graph. 

\subsubsection{Kendall's Tau In/Out correlation}{\label{sssc:kendall}}

In order to compute Kendall's Tau, we need to determine the number of concordant and discordant pairs. 
Starting with $G_n^a$, we observe that we have three kinds of joint observations, namely 
\begin{align*}
	I &: \left(D^-(e_{i \ast}), D^+(e_i^\ast)\right), \\
	II &: \left(D^-(f_{j \ast}), D^+(f_j^\ast)\right) \text{ and} \\
	III &: \left(D^-(g_\ast), D^+(g^\ast)\right).
\end{align*}
The combinations I and III, and II and III are concordant while I and II are discordant. It follows 
that $\mathscr{N}_c = (a + 1)n$ while $\mathscr{N}_d = an^2$. Hence we get, see Definition~\ref{def:kendalltau}, 
\[
	\kendall^+_-(G^a_n) = \frac{2(a + 1)n - 2an^2}{(a + 1)^2n^2 + (a + 1)n},
\]	
which gives $\nlim \kendall^+_-(G^a_n) = -\frac{2a}{(a + 1)^2}$. We observe that this equals $\nlim \Exp{
\spearman_-^+(G^a_n)}$, calculated in the previous section.

For the graph $\hat{G}^a_n$ we have four kinds of joint observations:
\begin{align*}
	I &: \left(\din(e_{i \ast}), \dout(e_i^\ast)\right), \\
	II &: \left(\din(f_{j \ast}), \dout(f_j^\ast)\right), \\
	III &: \left(\din(g_{1 \ast}), \dout(g_{1}^\ast)\right) \text{ and} \\
	IV &: \left(\din(g_{2 \ast}), \dout(g_{2}^\ast)\right). 
\end{align*}
Again the combinations I and II are discordant, while now I and III, and II and IV are concordant. 
Therefore we get $\mathscr{N}_c = (a + 1)n$ and $\mathscr{N}_d = an^2$, hence $\nlim \kendall^+_{-}
(\hat{G}^a_n) = -\frac{2a}{(a + 1)^2}$ which equals the limit for $\kendall^+_-(G^a_n)$. This is because 
the tied values, which are the majority in this example, make the influence of the extra node on the 
Kendall's tau negligible.

Note that $\nlim \kendall^+_-(G^a_n)$ decreases when we increase $a$. This is because the number
of tied values among the degrees increases with $a$. We already mentioned that $\kendall^\beta_\alpha$
gives smaller values when more ties are involved. Here this behavior is clearly present. 

\subsection{A collection of random In/Out bridge graphs}\label{ssec:randombridgegraph}

Let us now consider a collection of In/Out bridge graphs $G(W, Z)$ as defined in 
Section~\ref{ssec:bridgegraph}, where the values of $W$ and $Z$ are integer regularly 
varying random variables. 

Let $X, Y \in \mathcal{R}_{-\gamma}$ be independent and integer valued and fix $a \in \R_{> 0}$. For 
each $n \in \N$ take $(X_i)_{1 \le i \le n}$ and $(Y_i)_{1 \le i \le n}$ to be i.i.d. copies of $X$ 
and $Y$, respectively, and define $W_i = X_i + Y_i$ and $Z_i = \lfloor X_i + aY_i\rfloor$. Then we 
define the graph $\mathcal{G}^a_n$ as a disconnected collection of the graphs $(G(W_i, Z_i))_{1 \le 
i \le n}$. We will calculate $\pearson^+_-(\mathcal{G}^a_n)$ and prove that it converges to a random 
variable, which can have support on $(\varepsilon, 1)$ for any $\varepsilon\in (0,1]$ depending on a 
specific choice of $a$. 

Using the calculations in Section~\ref{sssec:pearson} we obtain:
\begin{align*}
	&\sum_{e \in E^a_n} D^-(e_\ast)D^+(e^\ast) = \sum_{i = 1}^n \left(X_i^2 + aY_i^2 + (1 + a)X_iY_i\right), \\
	&\sum_{v \in V^a_n} D^-(v)D^+(v) = \sum_{i = 1}^n \left(2X_i + (1 + a)Y_i\right), \\
	&\sum_{v \in V^a_n} D^-(v)^2D^+(v) = \sum_{i = 1}^n \left(X_i^2 + Y_i^2 + 2X_iY_i + X_i + aY_i\right), \\
	&\sum_{v \in V^a_n} D^-(v)D^+(v)^2 = \sum_{i = 1}^n \left(X_i^2 + a^2Y_i^2 + 2aX_iY_i + X_i + Y_i\right) \text{ and} \\
	&|E^a_n| = \sum_{i = 1}^n \left(2X_i + (1 + a)Y_i + 1\right).
\end{align*}
By the stable limit law we have a sequence $\seq{a}$ such that 
\[
	\frac{1}{a_n} \sum_{i = 1}^n X_i^2 \dlim S_X \quad \text{and} \quad \frac{1}{a_n} \sum_{i = 1}^n Y_i^2 \dlim S_Y 
	\quad \text{as } n \to \infty,
\] 
where $S_X$ and $S_Y$ are stable random variables. Further, due to Lemma 2.2 in~\cite{Litvak2012} we have
\[
	\frac{1}{a_n}\sum_{i = 1}^n X_iY_i \dlim 0, \quad \frac{1}{a_n} \sum_{i = 1}^n X_i \dlim 0 
	\quad \text{and} \quad \frac{1}{a_n} \sum_{i = 1}^n Y_i \dlim 0 \quad \text{as } n \to \infty.
\]
Combining this we get
\[
	\frac{1}{\sqrt{a_n}}\sigma_-(\mathcal{G}^a_n) \dlim \sqrt{S_X + S_Y}, \quad
	\frac{1}{\sqrt{a_n}}\sigma^+(\mathcal{G}^a_n) \dlim \sqrt{S_X + a^2S_Y} \quad \text{as } n \to \infty, 
\]
and hence
\[
	\pearson^+_-(\mathcal{G}^a_n) \dlim \frac{S_X + aS_Y}{\sqrt{\vphantom{^2}S_X + S_Y}\sqrt{S_X + a^2S_Y}}
	\quad \text{as } n \to \infty,
\]
which has support on $(0, 1)$. Now, take $0 < \varepsilon \le 1$ and consider the function $f(x) : 
(0, \infty) \to \R$ defined as
\[
	f(x) = \frac{1 + ax}{\sqrt{\vphantom{^2}1 + x}\sqrt{1 + a^2x}}.
\]
This function attains its minimum in $1/a$ and by solving $f(1/a) = \varepsilon$ for $a$
we get that for
\[
	a = \frac{2 - \varepsilon^2 \pm \sqrt{1 - \varepsilon}}{\varepsilon^2}
\]
this minimum equals $\varepsilon$. If we now introduce the random variable $T = S_Y/S_X$ we see
that for $a$ defined as above $\frac{1 + aT}{\sqrt{\vphantom{^2}1 + T}\sqrt{\vphantom{^2}1 + a^2T}}$ has
support contained in $(\varepsilon, 1)$.

This example shows that Pearson's correlation coefficients $\pearson^\beta_\alpha$ can converge to a 
non-negative random variable in the infinite size network limit. This behavior is undesirable for if we
consider two instances of the same model $\mathcal{G}^a_n$ then the values of $\pearson^+_-$ will be random
and hence could be very far apart. Therefore $\pearson^+_-$ is not suitable for measuring the In/Out 
correlation if we would like to find one number (population value) that characterizes the In/Out correlation
in this model.

\section{Experiments}\label{sec:experiments}

In this section we present experimental results for the degree-degree correlations introduced in 
Sections~\ref{sec:pearson} and~\ref{sec:rankcorrelations}. For the calculations we used the WebGraph
framework~\cite{Boldi2004, Boldi2004a} and the fastutil package from The Laboratory for Web 
Algorithmics (LAW) at the Università degli studi di Milano, \url{http://law.di.unimi.it}. The 
calculations were executed on the Wikepedia graphs, \url{http://wikipedia.org}, of nine different 
languages, obtained from the LAW dataset database. For each Wikipedia graph we calculated all four 
degree-degree correlations using the four measures introduced in this paper. 

The in- and out-degree distributions of these networks satisfy conditions of scale-free distributions 
with parameters between 1 and 2.5. Moreover, we evaluated the dependency between in- and out-degrees 
of the vertices, using angular measure~\cite[p. 313]{Resnick2007}, and found them to be independent. 
Therefore one could consider the Wikipedia networks as being generated by a model satisfying the 
conditions of Definition~\ref{def:graphsequencespace}.

In an attempt to quantify the results we compared them to a randomized setting. For this we did 20 
reconfigurations of the degree sequences of each graph, using the scheme described in Section 4.2 
of~\cite{chen2013}. More precisely, we used the \emph{erased directed configuration model}. In this 
scheme we first assign to each vertex $v$, $\dout(v)$ outbound stubs and $\din(v)$ inbound stubs. Then 
we randomly select an available outbound stub and combine it with a inbound stub, selected uniformly 
at random from all available inbound stubs, to make an edge. When this edge is a self loop we remove it. 
When we end up with multiple edges between two vertices we combine them into one edge. Proposition 4.2 
of~\cite{chen2013} now tells us that the distribution of the degrees of the resulting simple graph will, 
with high probability, be the the same as the original distribution. For each of these reconfigurations,  
all four types of degree-degree dependencies were evaluated using the four measures discussed above,
and then for each dependency type and each measure we took the average. The results are presented in 
Table~\ref{tbl:wikiresults}.

The first observation is that for each Wikipedia graph and dependency type, the measures $\spearman$, 
$\spearmanaverage$ and $\kendall$ have the same sign while $\pearson$ in many cases has a different sign. 
Furthermore, there are many cases where the absolute value of the three rank correlations is at least an
order of magnitude larger than that of Pearson's correlation coefficients. See for instance the Out/In 
correlations for DE, EN, FR and NL or the In/Out correlation for KO and RU. 

These examples illustrate the fact that Pearson's correlation coefficients are scaled down by the high 
variance in the degree sequences which in turn gave rise to Theorem~\ref{thm:convergencepearson}, while 
the rank correlations do not have this deficiency. Another interesting observation is that the values for 
$\spearman$ and $\spearmanaverage$ are almost in full agreement with each other. This would then suggest 
that, looking back at equation \eqref{eq:spearmanuniformaverage}, that $3\overline{\sigma}_\alpha
\overline{\sigma}^\beta \approx |E|^3 - |E|$ for the Wikipedia networks. Therefore one could freely change 
between these two when calculating degree-degree correlations. Note that $\spearman$ is somewhat 
computationally easier than $\spearmanaverage$ because there is no need to compute $\overline{\sigma}_\alpha
\overline{\sigma}^\beta$.

Finally, we notice that for the configuration model instances of the graphs, all correlation measures are 
close to zero, and the difference between different realizations of the model is remarkably small (see the 
values of $\sigma$). However, at this point very little can be said about statistical significance of these 
results because, as we proved above, $\pearson$ shows pathological behavior on large power law graphs and 
the setting of directed graphs is very different from the setting of independent observations. This raises 
important and challenging questions for future research: which magnitude of degree-degree dependencies 
should be seen as significant and how to construct mathematically sound statistical tests for establishing 
such significant dependencies.

\section{Discussion}\label{sec:discussion}

From Theorem~\ref{thm:convergencepearson} and the examples in Section~\ref{sec:examples} it is clear 
that Pearson's correlation coefficients have undesirable properties, based on their limiting behavior 
when the graph size goes to infinity. The question of whether or not rank correlations converge to 
correct population values in infinite graph size limit, has not been addressed in this paper, but it 
can be already answered affirmatively. For undirected graphs, it has been proved in \cite{Litvak2012},
and the results for directed graphs are the subject of our current research and will be presented in 
our upcoming paper~\cite{Hoorn2014}. This provides sufficient motivation for using such rank correlation
measures instead of Pearson's correlation coefficients for measuring degree-degree dependencies in 
directed networks with heavy-tailed degrees.

Nevertheless, we have also seen that when using rank correlations one needs to be careful when 
resolving the ties amongst the degrees. Furthermore, Spearman's rho and Kendall's tau turn very skewed 
distributions into uniform ones, thus they do not detect the influence of important hubs, as we saw in 
the example of the $G_n^a$ graph in Section~\ref{ssec:bridgegraph}. Possibly, these measures should be 
considered in combination  with measures for extremal dependencies, such as angular measure. Angular 
measure for two vectors $(X_i)_{i=1,\ldots,n}$ and $(Y_i)_{i=1,\ldots,n}$ is a rank correlation measure 
that characterizes whether $X_i$ and $Y_i$ tend to attain extremely large values simultaneously. We 
used this measure to verify the independence between in- and out- degrees of a node in Wikipedia graphs.

There is also an intriguing question of whether the four types of dependencies are related to one 
another. For instance, it is reasonable to think that if the Out/In and Out/Out correlations are highly 
positive, then the other two must also be (highly) positive. Indeed, if we take a node $v$ with high 
in-degree then it tends to have nodes of high out-degree connecting to it. Hence, out-degree of $v$ tends 
to be high as well because of the high positive Out/Out dependency. Therefore, if $v$ connects to another 
node $w$, then $w$ tends to have large in- and out-degree implying positive In/In and In/Out dependencies. 
It is very interesting to understand what are the feasibility bounds for possible combinations of the 
four dependency types in terms of different correlation measures.

Finally, although the results from percolation theory and the analysis of network stability under attack
give some insights in the impact of degree assortativity, it remains an open question what specific values 
of degree-degree correlation measures mean for the topology of directed networks in general. This shows 
that there are still many fundamental questions regarding degree-degree correlations in scale-free 
directed graphs.

\par \bigskip

\noindent\textbf{Acknowledgments}
This work is supported by the EU-FET Open grant NADINE (288956).

\afterpage{%
	\clearpage
	\newgeometry{left=0pt, right=0pt}
	\begin{landscape}
		\begin{table}[t]%
			\centering
			\small
			\begin{tabular}{|c|c|r|rr|r|rr|r|rr|r|rr|}
			\hline
				& & \multicolumn{3}{|c|}{Pearson} & \multicolumn{3}{|c|}{Spearman uniform} 
				& \multicolumn{3}{|c|}{Spearman average} & \multicolumn{3}{|c|}{Kendall} \\
			\hline
				& & & \multicolumn{2}{|c|}{Randomized} & & \multicolumn{2}{|c|}{Randomized}
				& & \multicolumn{2}{|c|}{Randomized} & & \multicolumn{2}{|c|}{Randomized}\\
				\multicolumn{1}{|c|}{Graph} & \multicolumn{1}{|c|}{$\alpha$/$\beta$} 
				& \multicolumn{1}{|c|}{Data} & \multicolumn{1}{c}{$\mu$} & \multicolumn{1}{c|}{$\sigma$}
				& \multicolumn{1}{|c|}{Data} & \multicolumn{1}{c}{$\mu$} & \multicolumn{1}{c|}{$\sigma$}
				& \multicolumn{1}{|c|}{Data} & \multicolumn{1}{c}{$\mu$} & \multicolumn{1}{c|}{$\sigma$}
				& \multicolumn{1}{|c|}{Data} & \multicolumn{1}{c}{$\mu$} & \multicolumn{1}{c|}{$\sigma$} \\
			\hline
			\multirow{4}{*}{DE wiki}
				& +/- & -0.0552 & -0.0178 & 0.0001 
							& -0.1434 & -0.0059 & 0.0002 
							& -0.1435 & -0.0059 & 0.0002 
							& -0.0986 & -0.0038 & 0.0008 \\
				& -/+ & 0.0154 & -0.0030 & 0.0002 
							& 0.0481 & -0.0008 & 0.0002 
							& 0.0484 & -0.0008 & 0.0002 
							& 0.0.326 & -0.0005 & 0.0001 \\
				& +/+ & -0.0323 & -0.0091 & 0.0002 
							& -0.0640 & -0.0048 & 0.0002 
							& -0.0640 & -0.0048 & 0.0002 
							& -0.0446 & -0.0006 & 0.0001 \\
				& -/- & -0.0123 & -0.0060 & 0.0001 
							& 0.0119 & -0.0009 & 0.0002 
							& 0.0120 & -0.0009 & 0.0002 
							& 0.0074 & -0.0032 & 0.0001 \\
			\hline
			\multirow{4}{*}{EN wiki}
				& +/- & -0.0557 & -0.0180 & 0 
							& -0.1999 & -0.0064 & 0.0001 
							& -0.1999 & -0.0064 & 0.0001 
							& -0.1364 & -0.0043 & 0.0001 \\
				& -/+ & -0.0007 & -0.0015 & 0.0001 
							& 0.0239 & -0.0011 & 0.0001 
							& 0.0240 & -0.0011 & 0.0001 
							& 0.0163 & -0.0008 & 0.0001 \\
				& +/+ & -0.0713 & -0.0125 & 0.0001 
							& -0.0855 & -0.0053 & 0.0001 
							& -0.0855 & -0.0053 & 0.0001 
							& -0.0581 & -0.0035 & 0.0001 \\
				& -/- & -0.0074 & -0.0024 & 0.0001 
							& -0.0664 & -0.0013 & 0.0001 
							& -0.0666 & -0.0013 & 0.0001 
							& -0.0457 & -0.0009 & 0.0001 \\
			\hline
			\multirow{4}{*}{ES wiki}
				& +/- & -0.1031 & -0.0336 & 0.0002 
							& -0.1429 & -0.0186 & 0.0003 
							& -0.1429 & -0.0186 & 0.0003 
							& -0.0972 & -0.0126 & 0.0002 \\
				& -/+ & -0.0033 & -0.0071 & 0.0002 
							& -0.0407 & -0.0047 & 0.0003 
							& -0.0417 & -00048 & 0.0003 
							& -0.0294 & -0.0034 & 0.0002 \\
				& +/+ & -0.0272 & -0.0201 & 0.0002 
							& 0.0178 & -0.0125 & 0.0003 
							& 0.0178 & -0.0125 & 0.0003 
							& 0.0119 & -0.0084 & 0.0002 \\
				& -/- & -0.0262 & -0.0116 & 0.0001 
							& -0.1627 & -0.0071 & 0.0003 
							& -0.1669 & -0.0072 & 0.0003 
							& -0.1174 & -0.0051 & 0.0002 \\
			\hline
			\multirow{4}{*}{FR wiki}
				& +/- & -0.0536 & -0.0252 & 0.0001 
							& -0.1065 & -0.0123 & 0.0002 
							& -0.1065 & -0.0123 & 0.0002 
							& -0.0720 & -0.0083 & 0.0002 \\
				& -/+ & 0.0048 & -0.0031 & 0.0002 
							& 0.0119 & -0.0016 & 0.0003 
							& 0.0121 & -0.0016 & 0.0003 
							& 0.0085 & -0.0011 & 0.0002 \\
				& +/+ & -0.0512 & -0.0173 & 0.0002 
							& -0.0126 & -0.0093 & 0.0002 
							& -0.0126 & -0.0090 & 0.0015 
							& -0.0087 & -0.0063 & 0.0001 \\
				& -/- & -0.0094 & -0.0054 & 0.0001 
							& -0.0262 & -0.0021 & 0.0003 
							& -0.0267 & -0.0025 & 0.0015 
							& -0.0186 & -0.0015 & 0.0002 \\
			\hline
			\multirow{4}{*}{HU wiki}
				& +/- & -0.1048 & -0.0378 & 0.0003 
							& -0.1280 & -0.0220 & 0.0006 
							& -0.1280 & -0.0220 & 0.0006 
							& -0.0877 & -0.0148 & 0.0004 \\
				& -/+ & 0.0120 & -0.0056 & 0.0005 
							& 0.0525 & 0.0002 & 0.0005 
							& 0.0595 & 0 & 0.0006 
							& 0.0442 & 0 & 0.0004 \\
				& +/+ & -0.0579 & -0.0261 & 0.0005 
							& -0.0207 & -0.0157 & 0.0005 
							& -0.0207 & -0.0157 & 0.0004 
							& -0.0140 & -0.0107 & 0.0003 \\
				& -/- & -0.0279 & -0.0084 & 0.0004 
							& 0.0051 & 0.0004 & 0.0005 
							& 0.0060 & 0.0002 & 0.0006 
							& 0.0050 & -0.0001 & 0.0005 \\
			\hline
			\multirow{4}{*}{IT wiki}
				& +/- & -0.0711 & -0.0319 & 0.0001 
							& -0.0964 & -0.0158 & 0.0002 
							& -0.0964 & -0.0158 & 0.0002 
							& -0.0653 & -0.0106 & 0.0002 \\
				& -/+ & 0.0048 & -0.0031 & 0.0002 
							& 0.0468 & -0.0013 & 0.0002 
							& 0.0469 & -0.0013 & 0.0003 
							& 0.0319 & -0.0009 & 0.0002 \\
				& +/+ & -0.0704 & -0.0204 & 0.0002 
							& -0.0277 & -0.0121 & 0.0002 
							& -0.0277 & -0.0122 & 0.0002 
							& -0.0189 & -0.0081 & 0.0001 \\
				& -/- & -0.0115 & -0.0050 & 0.0001 
							& -0.0428 & -0.0016 & 0.0002 
							& -0.0429 & -0.0016 & 0.0002 
							& -0.0296 & -0.0011 & 0.0002 \\
			\hline
			\multirow{4}{*}{KO wiki}
				& +/- & -0.0805 & -0.0562 & 0.0004 
							& -0.2696 & -0.0476 & 0.0037 
							& -0.2722 & -0.0482 & 0.0038 
							& -0.1985 & -0.0328 & 0.0073 \\
				& -/+ & 0.0157 & -0.0009 & 0.0030 
							& 0.1760 & 0.0019 & 0.0046 
							& 0.2323 & 0.0034 & 0.0046 
							& 0.1902 & 0.0031 & 0.0035 \\
				& +/+ & -0.1697 & -0.0357 & 0.0035 
							& 0.0016 & -0.0267 & 0.0041 
							& 0.0191 & -0.0272 & 0.0040 
							& 0.0170 & 0.0298 & 0.0415 \\
				& -/- & -0.0138 & -0.0034 & 0.0015 
							& -0.0493 & 0.0062 & 0.0045 
							& -0.0618 & 0.0083 & 0.0042 
							& -0.0463 & 0.0065 & 0.0032 \\
			\hline
			\multirow{4}{*}{NL wiki}
				& +/- & -0.0585 & -0.0346 & 0.0001 
							& -0.3017 & -0.0211 & 0.0002 
							& -0.3018 & -0.0211 & 0.0002 
							& -0.2089 & -0.0142 & 0.0002 \\
				& -/+ & 0.0100 & -0.0025 & 0.0003 
							& 0.0727 & -0.0007 & 0.0003 
							& 0.0730 & -0.0007 & 0.0003 
							& 0.0504 & -0.0004 & 0.0003 \\
				& +/+ & -0.0628 & -0.0194 & 0.0001 
							& 0.0016 & -0.0104 & 0.0003 
							& 0.0016 & -0.0104 & 0.0003 
							& 0.0015 & -0.0070 & 0.0002 \\
				& -/- & -0.0233 & -0.0091 & 0.0001 
							& -0.1498 & -0.0019 & 0.0003 
							& -0.1505 & -0.0019 & 0.0003 
							& -0.1048 & -0.0013 & 0.0002 \\
			\hline
			\multirow{4}{*}{RU wiki}
				& +/- & -0.0911 & -0.0225 & 0.0004 
							& -0.1080 & -0.0093 & 0.0015 
							& -0.1084 & -0.0093 & 0.0015 
							& -0.0755 & -0.0064 & 0.0010 \\
				& -/+ & 0.0398 & -0.0006 & 0.0009 
							& 0.1977 & 0 & 0.0008 
							& 0.2200 & 0.0001 & 0.0009 
							& 0.1655 & 0.0001 & 0.0007 \\
				& +/+ & 0.0082 & -0.0038 & 0.0010 
							& 0.2472 & 0.0002 & 0.0015 
							& 0.2480 & 0.0001 & 0.0015 
							& 0.1736 & 0.0001 & 0.0010 \\
				& -/- & -0.0242 & -0.0030 & 0.0007 
							& 0.0236 & 0.0009 & 0.0011 
							& 0.0255 & 0.0007 & 0.0015 
							& 0.0187 & 0.0006 & 0.0007 \\
			\hline
			\end{tabular}
			\caption{Degree-degree correlations for Wikipedia graphs. The data in the columns Randomized correspond to
			the results for the reconfigurations of the given Wikipedia network.}
			\label{tbl:wikiresults}
		\end{table}
	\end{landscape}
}
\restoregeometry

\clearpage

\bibliographystyle{plain}
\bibliography{degreedegreecorrelations}

\begin{thebibliography}{10}

\bibitem{Albert2002}
R{\'e}ka Albert and Albert-L{\'a}szl{\'o} Barab{\'a}si.
\newblock Statistical mechanics of complex networks.
\newblock {\em Reviews of modern physics}, 74(1):47, 2002.

\bibitem{Boguna2002}
Mari{\'a}n Bogun{\'a} and Romualdo Pastor-Satorras.
\newblock Epidemic spreading in correlated complex networks.
\newblock {\em Physical Review E}, 66(4):047104, 2002.

\bibitem{Boldi2004}
Paolo Boldi and Sebastiano Vigna.
\newblock The webgraph framework i: compression techniques.
\newblock In {\em Proceedings of the 13th international conference on World
  Wide Web}, pages 595--602. ACM, 2004.

\bibitem{Boldi2004a}
Paolo Boldi and Sebastiano Vigna.
\newblock The webgraph framework ii: Codes for the world-wide web.
\newblock In {\em Data Compression Conference, 2004. Proceedings. DCC 2004},
  page 528. IEEE, 2004.

\bibitem{Brede2005}
Markus Brede and Sitabhra Sinha.
\newblock Assortative mixing by degree makes a network more unstable.
\newblock {\em arXiv preprint cond-mat/0507710}, 2005.

\bibitem{chen2013}
Ningyuan Chen and Mariana Olvera-Cravioto.
\newblock Directed random graphs with given degree distributions.
\newblock {\em Stochastic Systems}, 3(1):147--186, 2013.

\bibitem{Cline1986}
Daren~B.H. Cline.
\newblock Convolution tails, product tails and domains of attraction.
\newblock {\em Probability Theory and Related Fields}, 72(4):529--557, 1986.

\bibitem{Franciscis2011}
Sebastiano de~Franciscis, Samuel Johnson, and Joaqu{\'\i}n~J. Torres.
\newblock Enhancing neural-network performance via assortativity.
\newblock {\em Physical Review E}, 83(3):036114, 2011.

\bibitem{Foster2010}
Jacob~G. Foster, David~V. Foster, Peter Grassberger, and Maya Paczuski.
\newblock Edge direction and the structure of networks.
\newblock {\em Proceedings of the National Academy of Sciences},
  107(24):10815--10820, 2010.

\bibitem{Kaltenbrunner2011}
Andreas Kaltenbrunner, Gustavo Gonzalez, Ricard Ruiz De~Querol, and Yana
  Volkovich.
\newblock Comparative analysis of articulated and behavioural social networks
  in a social news sharing website.
\newblock {\em New Review of Hypermedia and Multimedia}, 17(3):243--266, 2011.

\bibitem{Kendall1938}
Maurice~G. Kendall.
\newblock A new measure of rank correlation.
\newblock {\em Biometrika}, 30(1/2):81--93, 1938.

\bibitem{Laniado2011}
David Laniado, Riccardo Tasso, Yana Volkovich, and Andreas Kaltenbrunner.
\newblock When the wikipedians talk: Network and tree structure of wikipedia
  discussion pages.
\newblock In {\em ICWSM}, 2011.

\bibitem{Litvak2012}
Nelly Litvak and Remco van~der Hofstad.
\newblock Degree-degree correlations in random graphs with heavy-tailed
  degrees.
\newblock {\em arXiv preprint arXiv:1202.3071}, 2012.
\newblock To appear in Internet Mathematics.

\bibitem{Litvak2013}
Nelly Litvak and Remco van~der Hofstad.
\newblock Uncovering disassortativity in large scale-free networks.
\newblock {\em Physical Review E}, 87(2):022801, 2013.

\bibitem{liu2014impact}
Xiao~Fan Liu and Chi~Kong Tse.
\newblock Impact of degree mixing pattern on consensus formation in social
  networks.
\newblock {\em Physica A: Statistical Mechanics and its Applications},
  407:1--6, 2014.

\bibitem{Newman2002}
Mark~E.J. Newman.
\newblock Assortative mixing in networks.
\newblock {\em Physical review letters}, 89(20):208701, 2002.

\bibitem{Newman2003}
Mark~E.J. Newman.
\newblock Mixing patterns in networks.
\newblock {\em Physical Review E}, 67(2):026126, 2003.

\bibitem{Newman2003a}
Mark~E.J. Newman.
\newblock The structure and function of complex networks.
\newblock {\em SIAM review}, 45(2):167--256, 2003.

\bibitem{Piraveenan2012}
Mahendra Piraveenan, Mikhail Prokopenko, and Albert Zomaya.
\newblock Assortative mixing in directed biological networks.
\newblock {\em IEEE/ACM Transactions on Computational Biology and
  Bioinformatics (TCBB)}, 9(1):66--78, 2012.

\bibitem{Piraveenan2009}
Mahendra Piraveenan, Mikhail Prokopenko, and Albert~Y. Zomaya.
\newblock Assortativeness and information in scale-free networks.
\newblock {\em The European Physical Journal B}, 67(3):291--300, 2009.

\bibitem{Resnick2007}
Sidney~I Resnick.
\newblock {\em Heavy-tail phenomena: probabilistic and statistical modeling}.
\newblock Springer, 2007.

\bibitem{Spearman1904}
Charles Spearman.
\newblock The proof and measurement of association between two things.
\newblock {\em The American journal of psychology}, 15(1):72--101, 1904.

\bibitem{Srivastava2012}
Animesh Srivastava, Bivas Mitra, Niloy Ganguly, and Fernando Peruani.
\newblock Correlations in complex networks under attack.
\newblock {\em Physical Review E}, 86(3):036106, 2012.

\bibitem{srivastava2011}
Animesh Srivastava, Bivas Mitra, Fernando Peruani, and Niloy Ganguly.
\newblock Attacks on correlated peer-to-peer networks: An analytical study.
\newblock pages 1076--1081, 2011.

\bibitem{Hoorn2014}
Pim van~der Hoorn and Nelly Litvak.
\newblock Convergence of rank based degree-degree correlations.
\newblock Current research to appear on ArXiv.

\bibitem{Vazquez2003}
Alexei V\'azquez and Yamir Moreno.
\newblock Resilience to damage of graphs with degree correlations.
\newblock {\em Phys. Rev. E}, 67:015101, Jan 2003.

\bibitem{Vigna2014}
Sebastiano Vigna.
\newblock A weighted correlation index for rankings with ties.
\newblock {\em arXiv preprint arXiv:1404.3325}, 2014.

\end{thebibliography}

\end{document}